\documentclass[11pt,letterpaper,reqno]{amsart}
\usepackage[shortlabels]{enumitem}
\usepackage{amsfonts}
\usepackage{amssymb}
\usepackage{amsthm}
\usepackage{amsmath}
\usepackage{amscd}
\usepackage{t1enc}
\usepackage{lmodern}
\usepackage{graphicx}
\usepackage{stmaryrd}
\usepackage{graphics}
\usepackage{pict2e}
\usepackage{epic}
\numberwithin{equation}{section}
\usepackage[margin=1.3in]{geometry}
\usepackage{epstopdf} 
\usepackage{bbm}
\usepackage{tikz-cd}
\usetikzlibrary{tikzmark,quotes}
\usepackage{mathrsfs}
\usepackage{xcolor}
\usepackage{printlen}
\usepackage{eqparbox}
\usepackage[numbers]{natbib}[sort]

\usepackage{microtype}
\usepackage[linktocpage=true,colorlinks=true,pagebackref,hyperindex]{hyperref}
\usepackage{relsize}
\usepackage[bbgreekl]{mathbbol}
\usepackage{letltxmacro}

\definecolor{burntumber}{rgb}{0.54, 0.2, 0.14}
\definecolor{coolblack}{rgb}{0.0, 0.18, 0.39}
\definecolor{mygreen}{rgb}{0.0, 0.4, 0.0}
\hypersetup{
	linkcolor=burntumber,
	citecolor=mygreen
}

\setlist[enumerate]{leftmargin=2pc}
\setlist[itemize]{leftmargin=2pc}

\makeatletter
\def\l@subsection{\@tocline{2}{0pt}{2pc}{5pc}{}}
\makeatother

\theoremstyle{plain}
\newtheorem{thm}{Theorem}[section]
\newtheorem{lemma}[thm]{Lemma}
\newtheorem{cor}[thm]{Corollary}
\newtheorem{prop}[thm]{Proposition}
\newtheorem{conj}[thm]{Conjecture}

\newtheorem{introthm}{Theorem}

 \theoremstyle{definition}
\newtheorem{Def}[thm]{Definition}
\newtheorem{rmk}[thm]{Remark}
\newtheorem{?}[thm]{Problem}

\DeclareSymbolFontAlphabet{\mathbb}{AMSb} %to ensure that the meaning of \mathbb does not change
\DeclareSymbolFontAlphabet{\mathbbl}{bbold}
\newcommand{\Prism}{{\mathlarger{\mathbbl{\Delta}}}}

\newcommand{\bA}{\mathbb{A}}
\newcommand{\bC}{\mathbb{C}}

\newcommand{\bG}{\mathbb{G}}
\newcommand{\bL}{\mathbb{L}}
\newcommand{\bM}{\mathbb{M}}
\newcommand{\bP}{\mathbb{P}}
\newcommand{\bQ}{\mathbb{Q}}
\newcommand{\bR}{\mathbb{R}}
\newcommand{\bZ}{\mathbb{Z}}

\newcommand{\sP}{\mathscr{P}}

\newcommand{\sS}{\mathscr{S}}

\newcommand{\sV}{\mathscr{V}}

\newcommand{\Hom}{{\rm{Hom}}}

\newcommand{\id}{\textup{id}}
\newcommand{\res}{\big|}
\newcommand{\overbar}[1]{\mkern 1.5mu\overline{\mkern-1.5mu#1\mkern-1.5mu}\mkern 1.5mu}

\newcommand{\Spec}{\operatorname{Spec}}
\newcommand{\Spf}{\operatorname{Spf}}

\newcommand{\Spa}{\operatorname{Spa}}
\newcommand{\Spd}{\operatorname{Spd}}
\newcommand{\Perf}{\textup{Perf}}

\newcommand{\red}{\textup{red}}

\newcommand{\GL}{\textup{GL}}
\newcommand{\Res}{\textup{Res}}
\newcommand{\bs}{\backslash}
\newcommand{\ab}{\textup{ab}}
\newcommand{\Perfk}{\textup{Perf}_k}
\newcommand{\bdot}{\boldsymbol{.}}
\newcommand{\bdtimes}{\buildrel{\bdot}\over\times}
\newcommand{\Frob}{\textup{Frob}}

\newcommand{\crys}{\textup{crys}}
\newcommand{\Loc}{\textup{Loc}}
\newcommand{\ul}{\underline}
\newcommand{\Isom}{\textup{Isom}}
\newcommand{\Aut}{\textup{Aut}}
\newcommand{\Tors}{\textup{Tors}}

\newcommand{\eE}{{\sf E}}
\newcommand{\eG}{{\sf G}}
\newcommand{\eH}{{\sf H}}
\newcommand{\eK}{{\sf K}}
\newcommand{\eT}{{\sf T}}

\newcommand{\eZ}{{\sf Z}}

\newcommand{\fM}{\mathfrak{M}}
\newcommand{\fS}{\mathfrak{S}}

\newcommand{\cE}{\mathcal{E}}

\newcommand{\cG}{\mathcal{G}}
\newcommand{\cI}{\mathcal{I}}
\newcommand{\cM}{\mathcal{M}}
\newcommand{\cO}{\mathcal{O}}
\newcommand{\cP}{\mathcal{P}}

\newcommand{\cT}{\mathcal{T}}
\newcommand{\cX}{\mathcal{X}}
\newcommand{\cY}{\mathcal{Y}}

\newcommand{\Sh}{\textup{Sh}}

\newcommand{\ad}{\textup{ad}}

\newcommand{\dR}{\textup{dR}}
\newcommand{\Gr}{\textup{Gr}}

\newcommand{\Sht}{\textup{Sht}}
\newcommand{\Rep}{\textup{Rep}}
\newcommand{\Vect}{\textup{Vect}}
\newcommand{\univ}{\textup{univ}}
\newcommand{\LT}{\textup{LT}}

\newcommand{\Gal}{\textup{Gal}}

\newcommand{\marginalfootnote}[1]{
	\footnote{#1}
	\marginpar[{\hfill\sf\thefootnote}]{{\sf\thefootnote}}}
\newcommand{\edit}[1]{\marginalfootnote{#1}}

\LetLtxMacro\Oldedit\edit

\newcommand{\EnableEdits}{%
	\LetLtxMacro\edit\Oldedit%
}

\EnableEdits

\tikzset{
	labl/.style={anchor=south, rotate=90, inner sep=.5mm}
}

\begin{document}
\begin{abstract}
	We prove the Pappas-Rapoport conjecture on the existence of canonical integral models of Shimura varieties with parahoric level structure in the case where the Shimura variety is defined by a torus. As an important ingredient, we show, using the Bhatt-Scholze theory of prismatic $F$-crystals, that there is a fully faithful functor from $\mathcal{G}$-valued crystalline representations of $\Gal(\bar{K}/K)$ to $\cG$-shtukas over $\Spd(\mathcal{O}_K)$, where $\cG$ is a parahoric group scheme over $\mathbb{Z}_p$ and $\mathcal{O}_K$ is the ring of integers in a $p$-adic field $K$. 
\end{abstract}

\title[Integral models for Shimura varieties of toral type]{Canonical integral models for Shimura varieties of toral type}

\author{Patrick Daniels}
\thanks{Research partially supported by NSF DMS-1840234}
\address{Mathematics and Statistics Department, Skidmore College, Saratoga Springs, NY 12866}

%Science \\ H-1117 Budapest
%\\ P\'{a}zm\'{a}ny P\'{e}ter s\'{e}t\'{a}ny 1/C \\ Hungary} 
%
\email{pdaniels@skidmore.edu}
%
% \subjclass[2010]{Primary: 05C??. Secondary: 05C??}
%
% \keywords{sample paper} 

%\begin{abstract} 
%\end{abstract}

\maketitle
\tableofcontents

\section{Introduction}

In a recent breakthrough, Pappas and Rapoport \cite{PR2021} have given conditions which uniquely characterize integral models of Shimura varieties when the level subgroup at $p$ is parahoric. The aim of this work is to show that integral models satisfying these conditions exist for Shimura varieties of toral type, i.e., those defined by a torus. 

%Suppose that $\eT$ is a torus over $\bQ$, and that $(\eT, X)$ is a Shimura datum. Then for any compact open subgroup $\eK \subset \eT(\bA_f)$, we have a Shimura variety $\Sh_\eK(\eT,X)$, whose $\bC$-points are given by the finite set
%\begin{align*}
%	\Sh_\eK(\eT,X) = \eT(\bQ)\bs X \times \eT(\bA_f) / \eK.
%\end{align*}

%Shimura varieties are geometric objects with enormous arithmetic significance. For example, it is predicted by Langland's philosophy that the 

Before we go into detail, let us give some background on the study of integral models of Shimura varieties. When the level subgroup at $p$ of a Shimura variety is hyperspecial, it is expected that integral models with smooth reduction at $p$ exist (see \cite{Langlands1976}, for example). In that case, Milne observed \cite{Milne1992} that the collection of smooth integral models as the level away from $p$ varies can be uniquely characterized by an extension property, similar in nature to the N\'eron mapping property. Milne's conjecture was later refined by Moonen, who corrected the class of test schemes to be used in the extension property \cite{Moonen1998}. For Shimura varieties of abelian type, smooth integral models satisfying the extension property have been constructed by Kisin \cite{Kisin2010}, see also work of Vasiu \cite{Vasiu1999}. 

For many applications, one is primarily interested in the mod $p$ points of a given integral model of a Shimura variety, and it is from this interest that the necessity of a unique characterization of integral models arises. Indeed, different integral models may, a priori, give rise to different reductions modulo $p$. In order to avoid complications of this nature, one must start with a well chosen (or canonical) system of integral models.  

When the level at $p$ is more general than hyperspecial, however, the integral models can no longer always be expected to be smooth. Indeed, the singularities should be controlled by the associated local models, and these are known to be smooth only when the level structure is hyperspecial, and in certain exceptional cases, see \cite{HPR}. This type of pathology even occurs in some simple cases, such as the case of $\GL_2$ with Iwahori level structure, see, e.g., \cite[\textsection 1]{Rapoport2005}. Without smoothness, a characterization of integral models is not straightforward.

Pappas and Rapoport's \cite{PR2021} key observation is that Scholze's theory of $p$-adic shtukas over perfectoid spaces \cite{SW2020} can be extended to arbitrary $\bZ_p$-schemes. When the level at $p$ is parahoric, the problem of characterizing integral models can be solved by requiring the existence of a shtuka which lives over the integral model and satisfies certain properties. In the case of Hodge type Shimura data, Pappas and Rapoport construct integral models which admit such a shtuka.

To be more precise, let $(\eG, X)$ be a Shimura datum. Associated to every neat compact open subgroup $\eK \subset \eG(\bA_f)$ is the Shimura variety $\Sh_\eK(\eG,X)$, which admits a canonical model over a number field $\eE$. Fix a parahoric subgroup $\eK_p \subset \eG(\bQ_p)$ with associated $\bZ_p$-group scheme $\cG$. For the remainder of the introduction, we will consider only subgroups $\eK \subset \eG(\bA_f)$ which can be written $\eK=\eK_p\eK^p$ for some neat $\eK^p \subset \bG(\bA_f^p)$. Fix a place $v$ above $p$, and let $E$ be the completion of $\eE$ at $v$. We let $k$ denote an algebraic closure of the residue field of $E$. 

To simplify notation, throughout the introduction we assume the $\bQ$-split rank and the $\bR$-split rank of $\eG$ are equal; we work more generally in the body of the text. Using the work of Liu and Zhu \cite{LZ17}, Pappas and Rapoport show that there exists a $\cG$-shtuka $\sP_{\eK, E}$ over $(\Sh_\eK(\eG,X)_E)^\lozenge$ defined by the pro-\'etale $\cG(\bZ_p)$-cover
\begin{equation}\label{eq-prosystem}
	\Sh_{\eK_p'\eK^p}(\eG,X)_E \to \Sh_{\eK_p\eK^p}(\eG,X)_E
\end{equation}
as $\eK_p'$ varies over compact open subgroups of $\eK_p$. In this case, a $\cG$-shtuka $\sP$ over $(\Sh_\eK(\eG,X)_E)^\lozenge$ is a functorial rule which assigns to each untilt $S^\sharp$ of $S$ and map $f: S^\sharp \to \Sh_\eK(\eG,X)_E^\ad$ a $\cG$-shtuka $(\sP_S,\phi_{\sP_S})$ over $S$ with one leg at $S^\sharp$ in the sense of \cite{SW2020}.

A collection of integral models $(\sS_{\eK_p\eK^p})_{\eK^p}$ as $\eK^p$ varies over neat compact open subgroups of $\eG(\bA_f^p)$ is \textit{canonical} in the sense of Pappas and Rapoport if 
\begin{enumerate}[(a)]
	\item for every discrete valuation ring $R$ of mixed characteristic $(0,p)$, 
	\begin{align*}
		\left(\varprojlim_{\eK^p} \Sh_\eK(\eG,X)_E\right)(R[1/p]) = \left(\varprojlim_{\eK^p} \sS_\eK\right)(R),
	\end{align*}
	\item $\sP_{\eK,E}$ extends to a shtuka $\sP_\eK$ over $\sS_\eK$ for every $\eK$, and 
	\item for every point $x \in \sS_\eK(k)$, there exists an isomorphism
	\begin{align}\label{eq-isoms}
		\Theta_x: \widehat{\cM^\textup{int}_{\cG,b_x,\mu}}_{/x_0} \xrightarrow{\sim} 	(\widehat{\sS_\eK}_{/x})^\lozenge,
	\end{align}
	where $\cM^\textup{int}_{\cG,b_x,\mu}$ denotes the corresponding integral local Shimura variety, such that $\sP_\eK$ pulls back along $\Theta_x$ to the universal shtuka on $\widehat{\cM^\textup{int}_{\cG,b_x,\mu}}_{/x_0}$ given by its definition as a moduli space of shtukas.
\end{enumerate}
See Section \ref{section-conj} for a detailed explanation of these conditions and of the notation above. Pappas and Rapoport show that a system of normal and flat integral models with finite \'etale transition maps satisfying these conditions is uniquely determined, following a similar argument in Pappas's earlier work \cite{Pappas2022}.

Our focus in this paper is on the toral case, so suppose $\eG = \eT$ is a torus, let $T = \eT_{\bQ_p}$, and let $\eK_p = \cT(\bZ_p)$ be the unique parahoric subgroup of $T(\bQ_p)$ with corresponding $\bZ_p$-model $\cT$. In this case, the corresponding Shimura varieties are zero dimensional, that is 
\begin{align}\label{eq-canonicalmodel}
	\Sh_\eK(\eT,X)_E = \coprod \Spec(E_i),
\end{align}
where each $E_i$ is a finite extension of the local reflex field $E$. Then $\Sh_\eK(\eT,X)_E$ has an obvious integral model over $\cO_E$,
\begin{align*}
	\sS_\eK = \coprod \Spec(\cO_{E_i}),
\end{align*}
where $\cO_{E_i}$ is the ring of integers in $E_i$. The following is Theorem \ref{mainthm} below.
%In this case, the given Shimura varieties are very simple, and so are their integral models. Surprisingly, it is non-trivial to show the above conditions

\begin{introthm}\label{thm-A}
	The collection of integral models $(\sS_\eK)_{\eK^p}$ defined above is canonical in the sense of Pappas and Rapoport.
\end{introthm}

Surprisingly, although the Shimura varieties and their integral models are quite simple, the proof of Theorem \ref{thm-A} is not trivial. The key point is that $p$-adic shtukas over the $v$-sheaves associated to $\bZ_p$-schemes are complicated objects, and to define such an object requires some machinery. 

In the case of Shimura varieties of Hodge-type, the machinery is provided by the close connection between Shimura varieties of Hodge type and the Siegel moduli space, see \cite{PR2021}. In particular, the shtuka can roughly be constructed out of the Tate module of the pull-back of the universal abelian scheme over the Siegel space. In the toral case, we do not have access to any such moduli description, and instead, the machinery we need is provided by the following theorem (Theorem \ref{technicalthm} below), which is the key technical result of this work.

\begin{introthm}\label{thm-B}
	Let $K$ be a complete discretely valued extension of $\bQ_p$ with ring of integers $\cO_K$, and let $\cG$ be a parahoric group scheme over $\bZ_p$. Then there is a fully faithful functor
	\begin{align*}
		\left(
		\cG\textup{-valued crystalline representations of $\Gal(\bar{K}/K)$}
		\right)
		\to
		\left( \cG\textup{-shtukas over $\Spd(\cO_K)$}\right)
	\end{align*}
\end{introthm}

Before we explain what goes into the proof of Theorem \ref{thm-B}, let us explain how it can be used to construct the required shtuka. To construct a shtuka over $(\sS_\eK)^\lozenge$, it is enough to do so over $\Spd(\cO_{E_i})$ for each $i$. But in our case, the $\cT$-valued representation of $\Gamma_{E_i}$ obtained by from the pro-\'etale $\cT(\bZ_p)$-local system (\ref{eq-prosystem}) is crystalline, by class field theory and the definition of the canonical model (\ref{eq-canonicalmodel}) over $E$, so Theorem \ref{thm-B} provides us with such a shtuka. %One then must check that this construction recovers that of Pappas and Rapoport in the generic fiber, see Lemma \ref{lem-genericfiber}.

Theorem \ref{thm-B} itself has its roots in the Bhatt-Scholze theory of prismatic $F$-crystals \cite{BS21}. The main theorem of \cite{BS21} states that, for $K$ as in Theorem \ref{thm-B}, there is an equivalence of categories
\begin{align}\label{eq-BS}
	\left(
	\begin{array}{c}
	\textup{prismatic $F$-crystals} \\ \textup{over $\Spf(\cO_K)_\Prism$}
	\end{array}
	\right) 
	\xrightarrow{\sim} 
	\left(
	\begin{array}{c}
	\textup{$\bZ_p$-lattices in crystalline} \\ \textup{representations of $\Gal(\bar{K}/K)$} 
	\end{array}
	\right).
\end{align}
%explain something about what prismatic $F$-crystals are
One can obtain a shtuka over $\Spd(\cO_K)$ from a prismatic $F$-crystal over $\Spf(\cO_K)_\Prism$; this process is outlined in \cite[\textsection 4.4]{PR2021}. Our main innovations in this direction are twofold. First of all, we show that the resulting functor from prismatic $F$-crystals to shtukas is fully faithful, which provides a link between the two theories which may be interesting beyond the scope of this paper, see Theorem \ref{thm-fullyfaithful} and Corollary \ref{cor-fullyfaithful}. Second, we provide the group theoretic analog stated in Theorem \ref{thm-B}, see Proposition \ref{prop-exacttensor} and Theorem \ref{technicalthm}. 

We note that the existence of the isomorphisms (\ref{eq-isoms}) follows immediately from the theory in \cite{PR2022}, so the difficult part is showing that the shtukas pull back in the desired manner. For this we exploit the functoriality in the Shimura datum to reduce to the case of the Lubin-Tate tower, where the result essentially follows from the theory of Lubin-Tate formal groups, see Section \ref{sub-proof}.

In fact, our construction establishes the existence of a ``prismatic realization functor'' in the toral case (using the terminology of \cite{IKY}). This establishes the conjecture prismatic refinement of the universal $\cG$-shtuka as described in \cite[\textsection 4.4]{PR2021}. We remark that, subsequent to the appearance of a first draft of this work, a prismatic realization functor was constructed in \cite{IKY} for all abelian-type Shimura varieties with hyperspecial level structure. This is used to prove the Pappas-Rapoport conjecture in those cases. 

Finally, let us briefly mention that while there is some overlap between the toral-type and Hodge-type cases, the two situations are quite different in general. In the cases of Shimura varieties which are of both Hodge and toral type, it follows a posteriori from Theorem \ref{thm-A} that the integral models described here agree with those defined in \cite{KP2018}, \cite{KZ21}, and \cite{PR2021}, since all of these models are canonical in the sense of Pappas and Rapoport.

%The outline of the paper is as follows. In Section \ref{section-prelim} we review prelimaries from $p$-adic geometry and the work of Pappas and Rapoport, with a particular focus on the theory of shtukas. Section \ref{section-prismatic} is the main technical section which provides the inputs necessary for the proof of Theorem \ref{thm-A}. There we briefly review the prismatic theory, provide the connection with the Pappas-Rapoport extension of the theory of shtukas, and prove Theorem \ref{thm-B}. In Section \ref{section-SV} we make some recollections from the theory of Shimura varieties, state the Pappas-Rapoport conjecture on the existence of canonical models, and prove Theorem \ref{thm-A}. 
%Say something about the local completion part

%Mention the abelian type case

\subsection{Acknowledgments} 
I thank Alexander Bertoloni Meli, Ian Gleason, George Pappas, and Alex Youcis for helpful correspondences and conversations related to this project. I thank also the anonymous referee for their careful reading of this manuscript and for their many suggestions for improving the content and exposition in this article.

\subsection{Notation and conventions}
\begin{itemize}[leftmargin=*]
	\item If $X$ is a scheme over $A$, and $B$ is an $A$-algebra, we abbreviate $X \times_{\Spec(A)} \Spec(B)$ by $X \otimes_A B$.
	\item If $F$ is a $p$-adic local field, denote by $F^\textup{un}$ the maximal unramified extension of $F$ inside of a fixed algebraic closure $\bar{F}$, and we denote by $\breve{F}$ the completion of $F^\textup{un}$.

	%\item Let $\bA_f$ denote the finite adeles over $\bQ$, and for any prime $p$ let $\bA_f^p$ denote the finite adeles away from $p$.
	
	 %In this situation we write $\sigma$ for the relative Frobenius in $\textup{Gal}(\breve{F}/F)$.
	 \item If $(R,R^+)$ is a Huber pair with $R^+ = R^\circ$, we often write $\Spa(R)$ instead of $\Spa(R,R^+)$.
	%\item Let $S$ be either a nonarchimedean local field $E$ or its ring of integers $\mathcal{O}_E$, and let $T$ be $\Spa(E)$, or $\Spa(\mathcal{O}_E)$, respectively. If $X$ is a scheme which is separated and of finite type over $S$, we denote by $X^\ad$ the adic space $X \times_{S} T$, as defined in \cite[Prop. 3.8]{Huber1994}.
	\item For any field $K$ with fixed separable closure $\bar{K}$, we write $\Gamma_K$ for the absolute Galois group $\Gal(\bar{K}/K)$. %When $K$ is a $p$-adic field, we write $\Gamma_{K,0}$ for the inertia subgroup of $\Gamma_K$.
	\item If $X$ is a scheme or adic space, we  denote by $\Vect(X)$ the category of vector bundles on $X$. If $X = \Spec(A)$ is an affine scheme we write $\Vect(A)$ instead of $\Vect(\Spec(A))$.
%	\item If $G$ is a flat affine group scheme over a ring $A$, we denote by $\Rep_A(G)$ the category of algebraic representations $\rho: G \to \GL(\Lambda)$ of $G$ on finite projective $A$-modules $\Lambda$. 
%	and by 
%	\begin{align*}
%		\mathbbm{1}_B: \Rep_A(G) \to \Vect(B) 
%	\end{align*}
%	the forgetful functor $(\Lambda,\rho)\mapsto \Lambda\otimes_A B$ for any $A$-algebra $B$. 
%	\item Denote by Perfd the category of perfectoid spaces, and Perf the subcategory of perfectoid spaces in characteristic $p$.
	\item For a perfect field $k$ in characteristic $p$, we denote by $\Perfk$ the category of perfectoid spaces over $k$.
%	\item Denote by $\Schpf_{\bF_p}$ the category of perfect schemes in characteristic $p$, and $\Aff\Schpf_{\bF_p}$ the category of perfect affine schemes in characteristic $p$.
%	\item Denote by $\textup{(Perf)}^\sim$ the category of small $v$-sheaves on Perf. 
%	\item Denote by $(\Schpf_{\bF_p})^\sim$ the category of small scheme-theoretic $v$-sheaves (see \cite[\textsection 1.3.1]{Gleason2020}).
%	\item Let $\mathcal{O}_E$ be the ring of integers in a discretly valued $p$-adic field $E$. Denote by $\fml_{\mathcal{O}_E}$ the category of normal formal schemes which are flat and locally formally of finite type over $\mathcal{O}_E$.
\end{itemize}

\section{Preliminaries} \label{section-prelim}

\subsection{Recollections on shtukas}
We begin by establishing notation from the theory of perfectoid spaces and $v$-sheaves. We refer the reader to \cite{SW2020} and \cite{Scholze2017} for comprehensive background. Fix a perfect field $k$ of characteristic $p$. 

In what follows we will work with schemes, formal schemes, and adic spaces defined over $\bZ_p$ and $\bQ_p$. We remark that all of the constructions below work when $\bZ_p$ is replaced by a complete discrete valuation ring with perfect residue field, and we will use these constructions freely in the text below. The main examples of interest will be the ring of integers in a finite extension of $\bQ_p$ or $\breve{\bQ}_p$.

We first recall the functor from \cite[\textsection 18]{SW2020} associating a $v$-sheaf to any adic space over $\Spa(\bZ_p)$. If $X$ is an adic space over $\Spa(\bZ_p)$, $X^\lozenge$ is the set-valued functor on $\Perfk$ given by 
\begin{align*}
	X^\lozenge(S) = \{(S^\sharp, f)\}/\textup{isom.}
\end{align*}
for any $S$ in $\Perfk$, where $S^\sharp$ is an untilt of $S$ and $f$ is a morphism of adic spaces $f: S^\sharp \to X$. By \cite[Lem. 18.1.1]{SW2020}, the functor $X^\lozenge$ defines a $v$-sheaf on $\Perfk$. If $X = \Spa(A, A^+)$ for a Huber pair $(A,A^+)$, we write $\Spd(A,A^+)$ for $X^\lozenge$, and when $A^+ = A^\circ$, we write $\Spd(A)$ instead of $\Spd(A,A^+)$. If $\mathfrak{X}$ is a formal scheme over $\Spf(\bZ_p)$ which is locally formally of finite type, then we have a corresponding adic space $\mathfrak{X}^\ad$ defined over $\Spa(\bZ_p)$, and we use the shorthand $\mathfrak{X}^\lozenge$ to mean $(\mathfrak{X}^\ad)^\lozenge$. 

If $\mathscr{X}$ is now a scheme over $\Spec(\bZ_p)$, we can associate to $\mathscr{X}$ a $v$-sheaf using two different methods. For this we follow \cite[\textsection 2.2]{AGLR}. First assume $\mathscr{X}= \Spec(A)$ is affine. Consider the following $v$-sheaves:
\begin{itemize}
	\item Let $\mathscr{X}^\diamond$ be the $v$-sheaf over $\Spd(\bZ_p)$ defined by assigning to $S = \Spa(R,R^+)$ the set of isomorphism classes of pairs $(S^\sharp, f)$, where $S^\sharp= \Spa(R^\sharp, R^{\sharp+})$ is an untilt of $S$ over $\bZ_p$ and $f: A \to R^{\sharp+}$ is a ring homomorphism.
	\item Let $\mathscr{X}^\lozenge$ be the $v$-sheaf over $\Spd(\bZ_p)$ defined by assigning to $S = \Spa(R,R^+)$ the set of isomorphism classes of pairs $(S^\sharp, f)$ where $S^\sharp = \Spa(R^\sharp, R^{\sharp+})$ is an untilt of $S$ over $\bZ_p$ and $f: A \to R^\sharp$ is a ring homomorphism.
\end{itemize}
We remark that in \cite{PR2021}, the notation $(-)^\blacklozenge$ is used in place of the small diamond symbol $(-)^\diamond$. Both $(-)^\lozenge$ and $(-)^\diamond$ glue to define functors from schemes over $\Spec(\bZ_p)$ to $v$-sheaves over $\Spd(\bZ_p)$. These functors are referred to as the ``small diamond'' and ``big diamond'' functors, respectively. For any scheme $\mathscr{X}$ over $\Spec(\bZ_p)$ there is a natural transformation
\begin{align}\label{eq-jX}
	j_\mathscr{X}: \mathscr{X}^\diamond \to \mathscr{X}^\lozenge.
\end{align} 
The morphism $j_\mathscr{X}$ is a monomorphism of $\mathscr{X}$ is separated over $\bZ_p$, an open immersion if $\mathscr{X}$ is separated and of finite type over $\bZ_p$, and an isomorphism if $\mathscr{X}$ is proper over $\bZ_p$. There are analogous constructions for schemes over $\bQ_p$, which we use freely in the text below.

When $\mathscr{X}$ is separated and of finite type over $\bZ_p$, we have alternative derivations for each of the two diamond functors by passing first to one of two adic spaces over $\Spa(\bZ_p)$ associated to $\mathscr{X}$, see \cite[Rem. 2.11]{AGLR}. 
\begin{itemize}
	\item Let $\widehat{\mathscr{X}}$ denote the $p$-adic completion of $\mathscr{X}$, which is a formal scheme over $\Spf(\bZ_p)$. Then there is a natural isomorphism 
	\begin{align*}
		\mathscr{X}^\diamond \xrightarrow{\sim} \widehat{\mathscr{X}}^\lozenge.
	\end{align*}
	\item Let $X^\ad$ denote the fiber product $X \times_{\Spec(\bZ_p)} \Spa(\bZ_p)$ in the sense of \cite[Prop. 3.8]{Huber1994}. Then there is a natural isomorphism 
	\begin{align*}
		\mathscr{X}^\lozenge \xrightarrow{\sim} (\mathscr{X}^\ad)^\lozenge.
	\end{align*}
\end{itemize}
For the benefit of the reader, let us spell out the construction $\mathscr{X}^\ad$ more explicitly. If $\mathscr{X} = \Spec(\bZ_p[x_1, \dots, x_n]/(f_1, \dots, f_k))$, then 
\begin{align*}
	\mathscr{X}^\ad = \varinjlim_r \Spa(\bZ_p\langle p^rx_1, \dots, p^rx_n\rangle / (f_1, \dots, f_k)).
\end{align*} 
This is in contrast to $\widehat{\mathscr{X}}^\ad$, which is given by
\begin{align*}
	\widehat{\mathscr{X}}^\ad = \Spa(\bZ_p\langle x_1, \dots, x_n\rangle / (f_1, \dots, f_k)).
\end{align*}
In general, there is a morphism
\begin{align}\label{eq-adimm}
	\widehat{\mathscr{X}}^\ad \to \mathscr{X}^\ad
\end{align}
which is an open embedding when $\mathscr{X}$ is separated and is an isomorphism if $\mathscr{X} \to \Spec(\bZ_p)$ is proper, see \cite[Rem. 4.6 (iv)]{Huber1994}. We can also obtain the natural transformation (\ref{eq-jX}) by applying the $(-)^\lozenge$-functor for adic spaces to (\ref{eq-adimm}). 

When $\mathscr{X}$ is separated and of finite type over $\bZ_p$, by applying base change along $\Spd(\bZ_p) \to \Spd(\bQ_p)$ to (\ref{eq-jX}), we obtain an open immersion 
\begin{align*}
	\mathscr{X}^\diamond\times_{\Spd(\bZ_p)} \Spd(\bQ_p) \hookrightarrow (\mathscr{X})^\lozenge \times_{\Spd(\bZ_p)} \Spd(\bQ_p) = (\mathscr{X}\times_{\Spec(\bZ_p)} \Spec(\bQ_p))^\lozenge,
\end{align*}
with the last equality following from the compatibility of the functor $(-)^\lozenge$ with products. Following \cite[Def. 2.1.9]{PR2021}, we can then define $\mathscr{X}^{\lozenge \slash}$ to be the coproduct
\begin{align}\label{eq-slashdef}
	\mathscr{X}^{\lozenge \slash} := \mathscr{X}^\diamond \sqcup_{(\mathscr{X}^\diamond\times_{\Spd(\bZ_p)} \Spd(\bQ_p))} (\mathscr{X} \times_{\Spec(\bZ_p)} \Spec(\bQ_p))^\lozenge
\end{align}
as a $v$-sheaf on $\Spd(\bZ_p)$. We remark that there is a natural map 
\begin{align}\label{eq-slash}
	\mathscr{X}^{\lozenge \slash} \to \mathscr{X}^\lozenge
\end{align}
factoring $\mathscr{X}^\diamond \to \mathscr{X}^\lozenge$, which also becomes an isomorphism when $\mathscr{X} \to \Spec(\bZ_p)$ is proper.

If $S$ is a perfectoid space over $k$, we write $S \bdtimes \Spa(\bZ_p)$ for the analytic adic space defined in \cite[Prop. 11.2.1]{SW2020}. By the proof of {\it loc. cit.}, $S \bdtimes \Spa(\bZ_p)$ is sousperfectoid, i.e., it is covered by rational open subsets $\Spa(R,R^+)$ where $R$ is sousperfectoid in the sense of \cite[Def. 6.3.1]{SW2020}. When $S = \Spa(R,R^+)$ is affinoid, $S \bdtimes \Spa(\bZ_p)$ is given by $S \bdtimes \Spa(\bZ_p)  = \Spa (W(R^+))\bs \{ [\varpi] = 0\}$, where $\varpi$ is a fixed pseudo-uniformizer in $R^+$ and $[\varpi]$ denotes the canonical lift $(\varpi, 0, 0, \dots) \in W(R^+)$. Define furthermore 
\begin{align*}
	\mathcal{Y}(R,R^+) = \Spa(W(R^+)) \bs \{[\varpi]=0, \, p=0\},
\end{align*}
which is also sousperfectoid \cite[Prop. 3.6]{Kedlaya2020}. For any $S = \Spa(R,R^+)$ in $\Perfk$, we have the function
\begin{align*}
	\kappa: | S \bdtimes \Spa(\bZ_p) | \to [0,\infty)
\end{align*}
such that $\kappa(x) = (\log| [\varpi](\tilde{x})|) / (\log | p(\tilde{x})|)$, where $\tilde{x}$ denotes the maximal generalization of $x \in |S \bdtimes \Spa(\bZ_p)|$, see \cite[Prop. II.1.16]{FS21}. For any interval $I = [a,b] \subset [0,\infty)$, where $a$ and $b$ are rational numbers, denote by $\mathcal{Y}_I(S)$ the open subset of $S\bdtimes \Spa(\bZ_p)$ corresponding to the interior of $\kappa^{-1}(I)$. The Frobenius homomorphism $W(R^+) \to W(R^+)$ induces morphisms
\begin{align*}
	\mathcal{Y}_{[r,\infty)}(S) \to \mathcal{Y}_{[pr,\infty)}(S) 
\end{align*}
which will be denoted $\Frob_S$ or, occasionally, $\phi$.

Let $S = \Spa(R, R^+)$ is in $\Perf_k$, and let $S^\sharp = \Spa(R^\sharp, R^{\sharp+})$ be an untilt over  $S$. Associated with $S^\sharp$ are the de Rham period rings $B_\dR^+(R^\sharp)$ and $B_\dR(R^\sharp)$.  Explicitly, if $\xi$ is a generator of the surjective homomorphism $\theta: W(R^+) \to R^{\sharp+}$ coming from the choice of $S^\sharp$, then $B_\dR^+(R^\sharp)$ is the $\xi$-adic completion of $W(R^+)[1/p]$ and $B_\dR(R^\sharp) = B_\dR^+(R^\sharp)[1/\xi]$. We note for future reference that $B_\dR^+(R^\sharp)$ is then the completion $\widehat{\cO}_{S \bdtimes \bZ_p, \ S^\sharp}$ of $S\bdtimes \bZ_p$ along the closed Cartier divisor of $S\bdtimes \bZ_p$ corresponding to the untilt $S^\sharp$ of $S$.

Now let $G$ be a reductive group scheme over $\bQ_p$, and suppose $\cG$ is a parahoric $\bZ_p$-model for $G$ in the sense of \cite{BTII}. Let us briefly review the definitions of $\cG$-shtukas, first over perfectoid spaces, and then over the $v$-sheaf associated to a locally Noetherian adic space over $\Spa(\bZ_p)$. Let us note that all notions in this section have an obvious vector bundle analog, which we use freely in the text below. We refer the reader \cite{SW2020} and \cite{PR2021} for details. 

Let $S = \Spa(R,R^+)$ be affinoid perfectoid space over $k$, and let $S^\sharp = \Spa(R^\sharp, R^{\sharp+})$ be an untilt of $S$ over $\bZ_p$. A \textit{$\cG$-shtuka over $S$ with one leg at $S^\sharp$} is a pair $(\sP, \phi_{\sP})$, where $\sP$ is a $\cG$-torsor over $S \bdtimes \Spa(\bZ_p)$, and $\phi_{\sP}$ is a $\cG$-torsor isomorphism
\begin{align*}
	\phi_\sP: \Frob_S^\ast(\sP)\res_{S \bdtimes \Spa(\bZ_p) \setminus S^\sharp} \xrightarrow{\sim} \sP\res_{S \bdtimes \Spa(\bZ_p) \setminus S^\sharp}
\end{align*}
which is meromorphic along the closed Cartier divisor $S^\sharp \subset S \bdtimes \Spa(\bZ_p)$ in the sense of \cite[Def. 5.3.5]{SW2020}.

If $\mu: \bG_{m,\bar{\bQ}_p} \to G_{\bar{\bQ}_p}$ is a $G(\bar{\bQ}_p)$-conjugacy class of cocharacters with field of definition $E$, then we say a $\cG$-shtuka $(\sP, \phi_{\sP})$ over $S$ with one leg at $S^\sharp$ is \textit{bounded by $\mu$} if the relative position of $\Frob^\ast(\sP)$ and $\sP$ at $S^\sharp$ with respect to $\phi_\sP$ is bounded by the $v$-sheaf local model $\bM^v_{\cG,\mu}$. See \cite[\textsection 2.3.4]{PR2021} for a detailed explanation of this condition. 

We will denote the pair $(\sP, \phi_\sP)$ simply by $\sP$ when it is clear that we are speaking of the $\cG$-shtuka and not just the $\cG$-torsor. Since untilts $S^\sharp$ of $S$ correspond to morphisms $S \to \Spd(\bZ_p)$, we occasionally refer to a $\cG$-shtuka with one leg at $S^\sharp$ as a \textit{$\cG$-shtuka over $S/\,\Spd(\bZ_p)$}, if the morphism $S \to \Spd(\bZ_p)$ is understood. By \cite[Prop. 19.5.3]{SW2020}, the notion of a $\cG$-shtuka can be defined over a (not necessarily affinoid) perfectoid space $S$ equipped with a morphism to $\Spd(\bZ_p)$, and $\cG$-shtukas over $S /\, \Spd(\bZ_p)$ form a stack for the $v$-topology.

Let $R$ be an integral perfectoid ring in the sense of \cite[Def. 17.5.1]{SW2020}, and let $\xi$ be a generator of $\ker(\theta: W(R^\flat) \to R)$. A $\cG$\textit{-Breuil-Kisin-Fargues-module} over $R$ is a pair $(\sP, \phi_{\sP})$ consisting of a $\cG$-torsor $\sP$ over $\Spec(W(R^\flat))$ together with an isomorphism
\begin{align*}
	\phi_{\sP}: \phi^\ast \sP[1/\xi] \xrightarrow{\sim} \sP[1/\xi].
\end{align*}
Hereafter we will refer to these as $\cG$-BKF-modules. Following \cite{PR2021}, if $S = \Spa(R,R^+)$ is perfectoid with untilt $S^\sharp = \Spa(R^\sharp, R^{\sharp+})$, then we define a $\cG$-BKF-module over $S$ with one leg at $S^\sharp$ to be a $\cG$-BKF-module over $R^{\sharp+}$. We note that $R^{\sharp+}$ is integral perfectoid in this case, see \cite[Rmk 2.2.3 (i)]{PR2021}. 

If $S = \Spa(R,R^+)$ is perfectoid with untilt $S^\sharp = \Spa(R^\sharp, R^{\sharp+})$, pulling back along the morphism $\cY_{[0,\infty)}(S) \to \Spec(W(R^+))$ of locally ringed spaces determines a functor 
\begin{align*}
	(\cG\textup{-BKF-modules over }S\textup{ with one leg at $S^\sharp$}) \to (\cG\textup{-shtukas over $S$ with one leg at $S^\sharp$}).
\end{align*}
We say a $\cG$-BKF-module over $S$ with one leg at $S^\sharp$ is bounded by $\mu$ if the corresponding $\cG$-shtuka is bounded by $\mu$.

Let $\mathcal{F}$ be a $v$-sheaf over $\Spd(\bZ_p)$. We close this section by defining shtukas over $\mathcal{F}$, following \cite[\textsection 2.3]{PR2021}. Consider the slice category $\Perf_{\bZ_p} / \mathcal{F}$ consisting of perfectoid spaces $S$ over $\Spd(\bZ_p)$ together with a map $S \to \mathcal{F}$ over $\Spd(\bZ_p)$. Then there is a category $\Sht_\cG$ fibered over $\Perf_{\bZ_p} / \mathcal{F}$ whose fiber over $S \to \mathcal{F}$ is the category of $\cG$-shtukas over $S$ with one leg at the untilt $S^\sharp$ corresponding to $S \to \Spd(\bZ_p)$. A \textit{$\cG$-shtuka} over $\mathcal{F} / \Spd(\bZ_p)$ is defined to be a Cartesian section of $\Sht_\cG \to \Perf_{\bZ_p} / \mathcal{F}$  (in the sense of \cite[Tag 07IV]{stacks-project}).

In other words, a $\cG$-shtuka over $\mathcal{F} / \Spd(\bZ_p)$ is a compatible collection of functors
\begin{align}\label{eq-collection}
	\mathcal{F}(S) \to (\cG\textup{-shtukas over }S), \ (S^\sharp, f) \mapsto (\sP_S, \phi_{\sP_S})
\end{align}
for every $S$ in $\Perfk$ such that $(\sP, \phi_{\sP_S})$ has one leg at $S^\sharp$.  Here we view $\mathcal{F}(S)$ as a category whose only morphisms are the identity morphisms. To make the compatibility condition explicit, suppose $\sP$ is a collection of functors as in (\ref{eq-collection}). Let $a: T \to S$ be a morphism in $\Perfk$, and let $(S^\sharp,f)$ be an untilt of $S$ over $\Spa(\cO_{\breve{E}})$. By the tilting equivalence, $S^\sharp$ determines an untilt $(T^\sharp,g)$ of $T$ and a morphism $a^\sharp: T^\sharp \to S^\sharp$. Then $\sP$ associates to $(S^\sharp, f)$ a $\cG$-shtuka $(\sP_S, \phi_{\sP_S})$ over $S$ with one leg at $S^\sharp$ and to $(T^\sharp, g)$ a $\cG$-shtuka $(\sP_T, \phi_{\sP_T})$ over $T$ with one leg at $T^\sharp$. The collection of functors $\sP$ is compatible if, for every $T$ and $S$ as above, there is a natural isomorphism $a^\ast(\sP_S, \phi_{\sP_S}) \xrightarrow{\sim} (\sP_T, \phi_{\sP_T})$ of $\cG$-shtukas over $T$ with one leg at $T^\sharp$. Moreover, this collection of isomorphisms should satisfy the obvious cocycle condition.

We denote by $\Sht_\cG(\mathcal{F} / \Spd(\bZ_p))$ the category of $\cG$-shtukas over $\mathcal{F} / \Spd(\bZ_p)$, and by $\Sht_{\cG,\mu}(\mathcal{F} / \Spd(\bZ_p))$ the full subcategory of those bounded by $\mu$, that is those collections $\sP$ such that $(\sP, \phi_\sP)$ is bounded by $\mu$ for every $S \in \Perf_k$. If the map $\mathcal{F} \to \Spd(\bZ_p)$ is clear from context, we write simply $\Sht_\cG(\mathcal{F})$. Moreover, we have the analogous category of vector bundle shtukas over $\mathcal{F} / \Spd(\bZ_p)$, which we denote by $\Sht(\mathcal{F} / \Spd(\bZ_p))$, or once again simply $\Sht(\mathcal{F})$ if $\mathcal{F} \to \Spd(\bZ_p)$ is understood. Henceforth, we will usually refer to vector bundle shtukas simply as ``shtukas''.

We close this section with a lemma which helps us determine when a $\cG$-shtuka is bounded by a minuscule cocharacter $\mu$. For any small $v$-sheaf $\mathcal{F}$, let $|\mathcal{F}|$ denote the underlying topological space of $\mathcal{F}$ as in \cite[Prop. 12.7]{Scholze2017}. If $\mathcal{F}$ is defined over $\Spd(\bZ_p)$ then, following \cite[\textsection 3.4]{PR2021}, we say $\mathcal{F}$ is \textit{topologically flat} if the topological space of its generic fiber $\mathcal{F}_\eta: = \mathcal{F}\times_{\Spd(\bZ_p)} \Spd(\bQ_p)$ is dense in $|\mathcal{F}|$. 

\begin{lemma}\label{lem-topflat}
	Let $\mu: \mathbb{G}_{m,E} \to G_E$ be a minuscule cocharacter defined over $E$. Let $\mathcal{F}$ be a small $v$-sheaf over $\Spd(\cO_E)$, and let $\sP$ be a $\cG$-shtuka on $\mathcal{F}$. If $\mathcal{F}$ is topologically flat, then $\sP$ is bounded by $\mu$ if and only if its restriction to $\mathcal{F}_\eta$ is bounded by $\mu$.
\end{lemma}
\begin{proof}
	Let $\sP_\eta$ denote the pullback of $\sP$ along $\Spd(E) \to \Spd(\cO_E)$. Then $\sP_\eta$ is certainly bounded by $\mu$ if $\sP$ is, so it remains to show the converse. We will do this by proving that ``boundedness by $\mu$'' is a closed condition. 
	
	Let $\Perfk / \Spd(\bZ_p)$ denote the category of perfectoid spaces over $k$ equipped with a morphism to $\Spd(\bZ_p)$ (i.e., equipped with an untilt), and let $\operatorname{Gr}_{\cG}$ denote the Beilinson-Drinfeld Grassmannian over $\Spd(\bZ_p)$. If $S = \Spa(R,R^+)$ is in $\Perfk / \Spd(\bZ_p)$, then $\operatorname{Gr}_\cG(S)$ parametrizes $\cG$ torsors over $\Spec(B_\dR^+(R^\sharp))$ along with a trivilization over $\Spec(B_\dR(R^\sharp))$ \cite[Prop. 20.3.2]{SW2020}. 
	
%	 Denote by $L^+\cG$ and $L\cG$ the group functors on $\Perf_k / \Spd(\bZ_p)$  defined by 
%	\begin{align*}
%		L^+\cG(R,R^+) = \cG(B_\dR^+(R^\sharp)) \ \text{ and } \ L\cG(R,R^+) = \cG(B_\dR(R^\sharp)).
%	\end{align*}
%	Then by \cite[Prop. 20.3.2]{SW2020}, $\operatorname{Gr}_\cG$ is isomorphic to the \'etale sheafification of the functor 
%	\begin{align*}
%		(R, R^+) \mapsto L\cG(R, R^+) / L^+\cG(R,R^+).
%	\end{align*} 
	
	We define the integral local Hecke stack $\operatorname{Hk}_\cG$ to be the stack on $\Perfk / \Spd(\bZ_p)$ which assigns to a perfectoid space $S$ in $\Perfk/ \Spd(\bZ_p)$ the isomorphism classes of tuples $(\cE_1, \cE_2, \alpha)$, where $\cE_1$ and $\cE_2$ are $\cG$-torsors on $\Spec(B_\dR^+(R^\sharp))$, and $\alpha$ is an isomorphism
	\begin{align*}
		\alpha: \cE_1 [1/\xi] \xrightarrow{\sim} \cE_2 [1/\xi].
	\end{align*}
	It follows from \cite[Prop. 20.3.2]{SW2020} that $\operatorname{Hk}_\cG(S)$ is the quotient $L^+\cG \bs \operatorname{Gr}_\cG$ for the \'etale topology, where $L^+\cG$ is the group functor $(R,R^+) \mapsto \cG(B_\dR^+(R^\sharp))$ on $\Perfk / \Spd(\bZ_p)$.
	
	Recall the $v$-sheaf local model $\bM_{\cG,\mu}^v$ \cite[\textsection 21.4]{SW2020}, which is the closure for the $v$-topology of the diamond associated to the flag variety $\mathscr{F}\ell_{G,\mu}$ of parabolic subgroups of $G$ of type $\mu$ inside of $\operatorname{Gr}_{\cG}$. Analogously, we denote by $\operatorname{Hk}_{\cG,\mu}$ the (left) quotient of $\bM^v_{\cG,\mu}$ by $L^+\cG$. Since $\bM^v_{\cG,\mu} \to \operatorname{Gr}_\cG$ is a closed immersion and $\operatorname{Gr}_\cG \to \operatorname{Hk}_\cG$ is a $v$-cover, it follows from \cite[Prop. 10.11 (i)]{Scholze2017} that $\operatorname{Hk}_{\cG,\mu} \to \operatorname{Hk}_\cG$ is a closed immersion as well. 
	
	Associated to $\sP$ we have a morphism $\mathcal{F} \to \operatorname{Hk}_\cG$ of $v$-sheaves over $\Spd(\bZ_p)$. Indeed, given a morphism $S = \Spa(R,R^+) \to \mathcal{F}$ over $\Spd(\bZ_p)$, we obtain a $\cG$-shtuka $\sP_S$ over $S$ with one leg at $S^\sharp$. By restricting both $\sP_S$ and $\Frob_S^\ast \sP_S$ to the completion of $S \bdtimes \bZ_p$ at $S^\sharp$, we obtain two $\cG$-torsors on $\Spec(B_\dR^+(R^\sharp))$ with an isomorphism between their restrictions to $\Spec(B_\dR(R^\sharp))$, i.e., an object in $\operatorname{Hk}_\cG$. 
	
	Define $\mathcal{F}_\mu = \operatorname{Hk}_{\cG,\mu} \times_{\operatorname{Hk}_\cG} \mathcal{F}$. Since $\operatorname{Hk}_{\cG,\mu} \to \operatorname{Hk}_\cG$ is a closed immersion, the same is true of $\mathcal{F}_\mu \to \mathcal{F}$. Moreover, $\mathcal{F}_\mu(S)$ describes maps $S \to \mathcal{F}$ for which the corresponding shtuka $\sP_S$ is bounded by $\mu$. By assumption $\mathcal{F}_\eta \subset \mathcal{F}_\mu$, so $|\mathcal{F}_\mu| = |\mathcal{F}|$ by topological flatness of $\mathcal{F}$. Moreover, $\mathcal{F}_\mu \to \mathcal{F}$ is quasicompact, since it is a closed immersion (see e.g., \cite[Rmk.  18.2]{Scholze2017}). Hence by \cite[Lem. 12.11]{Scholze2017}, $\mathcal{F}_\mu \to \mathcal{F}$ is a surjective morphism of $v$-stacks, and is therefore an isomorphism. It follows that $\sP$ is bounded by $\mu$ on all of $\mathcal{F}$.
	
\end{proof}

\subsection{Shtukas and local systems}
In this section we discuss the connection between shtukas and pro-\'etale $\bZ_p$-local systems as in \cite[\textsection 2]{PR2021}. Let $S$ be a perfectoid space, and let $S_{\textup{pro\'et}}$ denote the pro-\'etale topology for $S$ as in \cite[Def. 8.2.6]{SW2020}. If $H$ is a topological group, denote by $\ul{H}$ the sheaf for the pro-\'etale topology on $\Perfk$ defined by
\begin{align}\label{eq-constantsheaf}
	\ul{H}(S) = C^0(|S|, H),
\end{align}
where $|S|$ is the topological space underlying $S$ and $C^0(|S|,H)$ denotes the set of continuous functions $|S| \to H$. In particular, for any finite free $\bZ_p$-module $M$, we have an associated pro-\'etale sheaf $\ul{M}$ on $\Perfk$. A \textit{pro-\'etale $\bZ_p$-local system on $S$} is a sheaf $\bL$ of $\ul{\bZ_p}$-modules on $S_{\textup{pro\'et}}$ such that $\bL$ is locally isomorphic to $\ul{M}$ for some finite free $\bZ_p$-module $M$. 

Suppose now $X$ is a locally Noetherian adic space over $\bZ_p$, and write $X_{\textup{pro\'et}}$ for the pro-\'etale site of $X$ as in \cite[\textsection 5.1]{BMS1}. We have also a notion of pro-\'etale $\bZ_p$-local system on $X$. Denote by $\ul{\bZ_p} $ the inverse limit $\varprojlim \bZ_/p^n\bZ$ as sheaves on $X_{\textup{pro\'et}}$. Then a \textit{pro-\'etale $\bZ_p$-local system on $X$} is a sheaf $\bL$ of $\ul{\bZ_p}$-modules on $X_{\textup{pro\'et}}$ such $\bL$ is locally isomorphic to $\ul{\bZ_p} \otimes M$ for some finitely generated $\bZ_p$-module $M$. By \cite[Prop. 8.2]{Scholze2013}, this is equivalent to the standard notion of a lisse $\bZ_p$-sheaf on the \'etale site $X_{\textup{\'et}}$ of $X$. We will denote by $\Loc_{\bZ_p}(X)$ the resulting category of pro-\'etale $\bZ_p$-local systems on $X$. 

\begin{rmk}\label{rmk-proetaleLS}
	Using \cite[Def. 9.1.4]{KL} we may extend the definition of the pro-\'etale site for locally Noetherian adic spaces given in \cite[\textsection 5.1]{BMS1} to arbitrary preadic spaces (in the terminology of \cite{KL}), and in particular to perfectoid spaces. If $S$ is a perfectoid space, the resulting pro-\'etale site differs from the site $S_{\textup{pro\'et}}$ defined in \cite{SW2020}, but the two categories of $\bZ_p$-local systems defined using the different sites are equivalent. Indeed, this follows from the fact that descent data for finite \'etale morphisms over $S$ are effective in both pro-\'etale topologies. In particular, given a locally Noetherian adic space $X$ over $\bZ_p$, a pro-\'etale $\bZ_p$-local system $\bL$ on $X$, and a morphism $S \to X$ from a perfectoid space $S$ to $X$, we may pull back $\bL$ to a $\bZ_p$-local system on $S$ in the sense described above. 
\end{rmk}

Following \cite{PR2021}, we define a \textit{pro-\'etale $\bZ_p$-local system} on $X^\lozenge$ as a compatible system of functors
\begin{align*}
	X^\lozenge(S) \to \Loc_{\bZ_p}(S), \ (S^\sharp, f) \mapsto \bL_S
\end{align*}
for every $S$ in $\Perfk$. We denote the resulting category of pro-\'etale $\bZ_p$-local systems on $X^\lozenge$ by $\Loc_{\bZ_p}(X^\lozenge)$. Let $\bL$ be a pro-\'etale $\bZ_p$-local system on $X$, and let $S$ be a perfectoid space in $\Perfk$ with untilt $S^\sharp$. Then any morphism of adic spaces $S^\sharp \to X$ allows us to define a pro-\'etale $\bZ_p$-local system on $S$ by first pulling back $\bL$ to $S^\sharp$ (see Remark \ref{rmk-proetaleLS}) and then applying the tilting equivalence. This defines a functor
\begin{align} \label{eq-locsys}
	\Loc_{\bZ_p}(X) \to \Loc_{\bZ_p}(X^\lozenge).
\end{align}

\begin{lemma}\label{lem-locsys}
	If $X$ is an analytic adic space over $\bZ_p$, the functor $(\ref{eq-locsys})$ is an equivalence of categories. 
\end{lemma}
\begin{proof}
	We use pro-\'etale descent to construct a quasi-inverse functor. Let $\bL$ be a pro-\'etale $\bZ_p$-local system over $X^\lozenge$. By \cite[Lem. 15.3]{Scholze2017}, there is a perfectoid space $\widetilde{X}$ which provides a pro-\'etale cover $\psi: \widetilde{X} \to X$. Then $(\widetilde{X}, \psi)$ defines a point in $X^\lozenge(\widetilde{X}^\flat)$, so $\bL$ determines a pro-\'etale $\bZ_p$-local system $\bL_{\widetilde{X}^\flat}$ over $\widetilde{X}^\flat$, and by the tilting equivalence we obtain a pro-\'etale $\bZ_p$-local system $\bL_{\widetilde{X}}$ on $\widetilde{X}$. 
	
	The fiber product $\widetilde{X}^\flat \times_{X^\lozenge} \widetilde{X}^\flat$ is representable by a perfectoid space over $X^\lozenge$. Denote by $p_1$ and $p_2$ the two projection morphisms $\widetilde{X}^\flat \times_{X^\lozenge} \widetilde{X}^\flat \to \widetilde{X}^\flat$. Since $\psi\circ p_1 = \psi \circ p_2$, the transition isomorphisms for $\bL$ induce an isomorphism $\alpha^\flat: p_1^\ast (\bL_{\widetilde{X}^\flat}) \xrightarrow{\sim} p_2^\ast (\bL_{\widetilde{X}^\flat})$. Applying the tilting equivalence again, we obtain an isomorphism $\alpha: p_1^\ast (\bL_{\widetilde{X}}) \xrightarrow{\sim} p_2^\ast (\bL_{\widetilde{X}})$ of pro-\'etale $\bZ_p$-local systems on $\widetilde{X} \times_X \widetilde{X}$, which is a pro-\'etale descent datum by the cocycle condition for $\bL$. Hence we obtain a pro-\'etale local system $\bL_X$ on $X$ via descent. We leave it to the reader to check that this construction defines a quasi-inverse to (\ref{eq-locsys}).
\end{proof}

\begin{rmk}
	One could alternatively define $\bZ_p$-local systems on $X^\lozenge$ using the \'etale or quasi-pro-\'etale site of $X^\lozenge$ in the sense of \cite[Def. 14.1]{Scholze2017}. We choose the definition given above for consistency with \cite{PR2021} and because it is most easily seen to be compatible with the constructions we make below. In the end, the resulting category of $\bZ_p$-local systems on $X^\lozenge$ is independent of this choice. This can be seen by combining Lemma \ref{lem-locsys} with \cite[Lem. 15.6]{Scholze2017} and \cite[Prop. 3.7]{MannWerner}.
\end{rmk}

\begin{rmk} \label{rmk-forgetful}
	If $K$ is a complete discretely valued extension of $\bQ_p$, and $X = \Spa(K, \cO_K)$, then $\widetilde{X}$ can be chosen as $\Spa(C, \cO_C)$ for a complete algebraic closure $C$ of $K$. In that case, $\Loc_{\bZ_p}(X) = \Rep_{\bZ_p}(\Gamma_K)$, and, under the equivalence (\ref{eq-locsys}), evaluation on $\Spa(C,\cO_C)$  corresponds to the forgetful functor $\Rep_{\bZ_p}(\Gamma_K) \to \Vect(\bZ_p)$. 
\end{rmk}

Let $\cG$ is a smooth affine group scheme over $\bZ_p$, and let $\ul{\cG(\bZ_p)}$ be the pro-\'etale sheaf on $\Perfk$ associated to $\cG(\bZ_p)$ (see (\ref{eq-constantsheaf}). Denote by $\Tors_{\ul{\cG(\bZ_p)}}(S)$ the category of pro-\'etale $\ul{\cG(\bZ_p)}$-torsors on $S$ as in \cite[\textsection 9.3]{SW2020}.

A \textit{pro-\'etale $\ul{\cG(\bZ_p)}$-torsor} on $X^\lozenge$ is a compatible system of functors
\begin{align*}
	X^\lozenge(S) \to \Tors_{\ul{\cG(\bZ_p)}}(S), \ (S^\sharp, f) \mapsto \bP_S
\end{align*}
for every $S$ in $\Perfk$. The category $\Loc_{\bZ_p}(X^\lozenge)$ inherits exact and tensor structures from the categories $\Loc_{\bZ_p}(S)$ as $S$ varies in $\Perf_k$, hence we have a Tannakian interpretation of pro-\'etale $\ul{\cG(\bZ_p)}$-torsors as well. Let $\bP$ be a pro-\'etale $\ul{\cG(\bZ_p)}$-torsor over $X^\lozenge$, and let $\rho: \cG \to \GL(\Lambda)$ be an algebraic representation of $\cG$ on a finite $\bZ_p$-module $\Lambda$. For every $(S^\sharp, f) \in X^\lozenge(S)$, we obtain a pro-\'etale $\ul{\bZ_p}$-local system 
\begin{align*}
	\bL_\rho:=\bP_S \times^{\ul{\cG(\bZ_p)}} \ul{\Lambda}_S
\end{align*}
on $S$. In other words, $\bL_\rho$ is the quotient of $\bP_S \times \ul{\Lambda}_S$ by the $\ul{\cG(\bZ_p)}$-action $g \cdot (p, \lambda) = (pg^{-1}, g\lambda)$. 

Let $\Rep_{\bZ_p}(\cG)$ denote the category of algebraic representations $\rho: \cG \to \GL(\Lambda)$ of $\cG$ on finite projective $\bZ_p$-modules $\Lambda$. Thus, from a pro-\'etale $\ul{\cG(\bZ_p)}$-torsor $\bP$ we obtain an exact $\bZ_p$-linear tensor functor
\begin{align}\label{eq-tensorfunctor}
	\omega_\bP: \Rep_{\bZ_p}(\cG) \to \Loc_{\bZ_p}(X^\lozenge), \ (\Lambda, \rho) \mapsto \bL_\rho.
\end{align}

\begin{lemma}\label{lem-tannakiantorsors}
	The assignment $\bP \mapsto \omega_{\bP}$ determines an equivalence of categories between $\Tors_{\ul{\cG(\bZ_p)}}(X^\lozenge)$ and the category of exact $\bZ_p$-linear tensor functors $\Rep_{\bZ_p}(\cG) \to \Loc_{\bZ_p}(X^\lozenge)$.
\end{lemma}
\begin{proof}
	By functoriality of the construction, it is enough to show that if $S$ is in $\Perfk$, then the category of pro-\'etale $\ul{\cG(\bZ_p)}$-torsors on $S^\lozenge$ is equivalent to the category of exact $\bZ_p$-linear tensor functors $\Rep_{\bZ_p}(\cG) \to \Loc_{\bZ_p}(S^\lozenge)$. For such an $S$, we have $\Loc_{\bZ_p}(S^\lozenge) = \Loc_{\bZ_p}(S)$, and the result follows from the proof of \cite[Prop. 22.6.1]{SW2020}. Let us explain.
	
	Denote by $\omega_0$ the exact $\bZ_p$-linear tensor functor $\Rep_{\bZ_p}(\cG) \to \Loc_{\bZ_p}(S)$ given by $(\Lambda, \rho) \mapsto \ul{\Lambda}_S$. If $\omega$ is another exact $\bZ_p$-linear tensor functor $\omega: \Rep_{\bZ_p}(\cG) \to \Loc_{\bZ_p}(S)$, define a pro-\'etale sheaf $\bP_\omega = \Isom^\otimes(\omega_0, \omega)$ classifying tensor isomorphisms between $\omega_0$ and $\omega$. That is, the points of $\bP_\omega$ in $T\to S$ are given by isomorphisms of tensor functors $\omega_{0,T} \to \omega_T$, where the subscript $(-)_T$ indicates that we compose the given tensor functor with pullback $\Loc_{\bZ_p}(S) \to \Loc_{\bZ_p}(T)$. Similarly, define $\Aut^\otimes(\omega_0) = \Isom^\otimes(\omega_0,\omega_0).$
	
	The natural action of $\ul{\cG(\bZ_p)}$ on $\omega_0$ determines morphism of pro-\'etale sheaves $\ul{\cG(\bZ_p)} \to \Aut^\otimes(\omega_0)$, which can be seen to be an isomorphism by reducing to the case where $S$ is strictly totally disconnected and applying ordinary Tannakian duality. Thus we obtain an action of $\ul{\cG(\bZ_p)}$ on $\bP_\omega$ by pre-composition. The arguments of \cite[Prop. 22.6.1]{SW2020} imply that $\omega$ is pro-\'etale locally isomorphic to $\omega_0$, and therefore that $\bP_\omega$ is a pro-\'etale $\ul{\cG(\bZ_p)}$-torsor. We conclude by noting that $\bP \mapsto \omega_{\bP}$ and $\omega \mapsto \bP_{\omega}$ determine mutually quasi-inverse functors. 
\end{proof}

%As in the case of local systems, pullback and tilting define a natural functor
%\begin{align}\label{eq-torsors}
%	\Tors_{\ul{\cG(\bZ_p)}}(X) \to \Tors_{\ul{\cG(\bZ_p)}}(X^\lozenge).
%\end{align}
%
%\begin{lemma}
%	If $X$ is an analytic adic space over $\bZ_p$, the functor $(\ref{eq-torsors})$ is an equivalence of categories.
%\end{lemma}
%\begin{proof}
%	The proof for Lemma \ref{lem-locsys} works here as well, after replacing ``$\bZ_p$-local system'' by ``$\ul{\cG(\bZ_p)}$-torsor'' everywhere.
%\end{proof}

In the remainder of this section we explain how to assign a $\bZ_p$-local system to a shtuka over a $v$-sheaf, following \cite[\textsection 2.4]{PR2021} and \cite[Ch. 22]{SW2020}. Let $S = \Spa(R,R^+)$ be an affinoid perfectoid space over $k$ with an untilt $S^\sharp = \Spa(R^\sharp, R^{\sharp+})$ over $\cO_{\breve{E}}$, and suppose $S^\sharp$ has the property that the corresponding morphism $S \to \Spd(\cO_E)$ factors through $\Spd(E)$. Consider the integral Robba ring, as in \cite[Def. 4.2.2]{KL} and \cite[\textsection 22.3]{SW2020}:
\begin{align}\label{eq-Robba}
	\widetilde{\mathcal{R}}^\mathrm{int}_S := \varinjlim_{r > 0} H^0(\cY_{[0,r]}(S), \cO_{\cY_{[0,r]}(S)}).
\end{align}
This ring carries a natural Frobenius morphism, which is compatible with that of $S \bdtimes \bZ_p$. 

Now let $(\sV, \phi_{\sV})$ be a (vector bundle) shtuka over $S$ with one leg at $S^\sharp$. For any $r > 0$ with the property that $\cY_{[0,r]}(S) \subset S \bdtimes \bZ_p \setminus S^\sharp$, we obtain by restriction of $(\sV, \phi_{\sV})$ a $\phi^{-1}$-module on $\cY_{[0,r]}(S)$. Note that such an $r$ is guaranteed to exist because we have taken the untilt $S^\sharp$ to live over $E$. Passing to the limit, we obtain a $\phi$-module 
\begin{align*}
	\varinjlim_{r > 0} H^0(\cY_{[0,r]}(S), \sV)
\end{align*}
over $\widetilde{\mathcal{R}}^\mathrm{int}_S$. By \cite[Thm. 8.5.3, Thm. 9.3.7]{KL}, such an object is equivalent to a $\bZ_p$-local system on $S$. This construction defines a functor 
\begin{align}\label{eq-shtukaslocsys}
	(\textup{shtukas over $S$ with one leg at $S^\sharp$}) \to \Loc_{\bZ_p}(S).
\end{align}

Let us return to the notation of the previous section, so $G$ is a reductive group scheme over $\bQ_p$ and $\cG$ is a parahoric $\bZ_p$-model for $G$. 
Let $(\sP, \phi_{\sP})$ be a shtuka over $S$ with one leg at $S^\sharp$. Applying the Tannakian formalism to (\ref{eq-shtukaslocsys}), we obtain a functor
\begin{align}\label{eq-Gshtukastors}
	(\textup{$\cG$-shtukas on $S$ with one leg at $S^\sharp$}) \to \Tors_{\ul{\cG(\bZ_p)}}(S).
\end{align}

Since these constructions are functorial in $S$, they extend to functors for shtukas over $v$-sheaves associated to adic spaces.

\subsection{Crystalline representations and prismatic $F$-crystals}
% explain how to go from a lattice in a crystalline repn to a BK-module to a shtuka
Let $K$ be a complete discretely valued extension of $\bQ_p$ with ring of integers $\cO_K$ and perfect residue field $k$. Fix an algebraic closure $\bar{K}$ of $K$, and let $\Gamma_K = \Gal(\bar{K}/K)$ be the absolute Galois group of $K$. Denote by $C$ the completion of $\bar{K}$, and by $\cO_C$ its ring of integers.

Let $\Rep_{\bZ_p} (\Gamma_K)$ denote the category of finite free $\bZ_p$-modules $\Lambda$ equipped with a continuous action of $\Gamma_K$, and let $\Rep_{\bZ_p}^\crys( \Gamma_K)$ denote the full subcategory of $\Rep_{\bZ_p} (\Gamma_K)$ consisting of those representations such that $\Lambda[1/p]$ is crystalline. In this section we describe a method for obtaining a vector bundle shtuka over $\Spd(\cO_K)$ from a representation in $\Rep_{\bZ_p}^\crys(\Gamma_K)$. This method relies on the description of $\Rep_{\bZ_p}^\crys(\Gamma_K)$ as the category of prismatic $F$-crystals on $\Spf(\cO_K)_\Prism$ given in \cite{BS21}. For background on prisms and prismatic $F$-crystals we refer the reader to   \cite{BS22} and \cite{BS21}.

Let $\cX$ be a $p$-adic formal scheme. The \textit{absolute prismatic site of $\cX$}, denoted $\cX_\Prism$, is the site whose underlying category consists of the opposite category of the category of pairs $((A,I), x)$, where $(A,I)$ is a bounded prism and $x: \Spf(A/I) \to \cX$ is a map of formal schemes, and whose topology is given by the flat topology on prisms. Here we say a morphism of prisms $(A, I) \to (B, IB)$ is \textit{(faithfully) flat} if the ring homomorphism $A \to B$ is $(p,I)$-completely (faithfully) flat, that is, $B/(p,I)B$ is (faithfully) flat over $A/(p,I)$, and $\textup{Tor}_i^A(A/(p,I), B) = 0$ for all $i > 0$. 

\begin{rmk}\label{rmk-ff}
	If $(A,I) \to (B,IB)$ is a (faithfully) flat morphism of prisms, the induced maps $A/(p,I)^n \to B/(p,I)^nB$ are (faithfully) flat for every $n$. Indeed, flatness follows from \cite[Tag 051C]{stacks-project}, and then faithful flatness follows by induction since the maps $A/(p,I)^{n} \to A/(p,I)^{n-1}$ are nilpotent thickenings.
\end{rmk}

If $((A,I),x)$ is an object in the underlying category of $\cX_\Prism$, we will sometimes say that $(A,I)$ is a \textit{prism over $\cX$}. The site $\cX_\Prism$ is naturally equipped with a structure sheaf $\cO_\Prism: ((A,I),x) \mapsto A$ and an ideal sheaf $\cI_\Prism: ((A,I),x) \mapsto I$. 

Following \cite{BS21}, we write $\Vect^\varphi(\cX_\Prism, \cO_\Prism)$ for the category of \textit{prismatic $F$-crystals} on $\cX_\Prism$. This is the category of pairs $(\cE, \varphi_\cE)$ consisting of a vector bundle $\cE$ on the ringed site $(\cX_\Prism, \cO_\Prism)$ along with an isomorphism $\varphi_\cE: \varphi^\ast \cE[1/\cI_\Prism] \xrightarrow{\sim} \cE[1/\cI_\Prism]$. By \cite[Prop. 2.7]{BS21}, a prismatic $F$-crystal $(\cE, \varphi_\cE)$ is given concretely by the following data: For every for every object $\tilde{x} = ((A,I), x)$ in $\cX_\Prism$, $(\cE, \varphi_{\cE})$ determines a pair $(\cE_A, \varphi_{\cE_A})$ consisting of a vector bundle $\cE_A$ on $\Spec(A)$ along with an isomorphism 
\begin{align*}
	\varphi_{\cE_A}: \varphi_A^\ast\left(\cE_A\right)\res_{\Spec(A[1/I])} \xrightarrow{\sim} \cE_A\res_{\Spec(A[1/I])}.
\end{align*}
These come equipped with transition isomorphisms
\begin{align*}
	\theta_f:\cE_A \otimes_A B \xrightarrow{\sim} \cE_B,
\end{align*}
which are compatible with Frobenius, for every morphism $f: ((A,I),x) \to ((B,J),y)$ in $\cX_\Prism$. Moreover the collection $\{\theta_f\}$ satisfies the obvious cocycle condition.% Note that the pair $(\cE_A, \varphi_{\cE_A})$ actually depends on the ideal $I$ and the map $x:\Spf(A/I) \to \cX$, not just on $A$. We think this abuse of notation is unlikely to cause confusion.

As in \cite{BS21}, we write $\cO_\Prism[1/\cI_\Prism]$ for the sheaf of rings $(A,I) \mapsto A[1/I]$, and we define $\cO_\Prism[1/\cI_\Prism]_p^\wedge$ to be 
\begin{align*}
	\cO_\Prism[1/\cI_\Prism]_p^\wedge := \varprojlim_n \cO_\Prism[1/\cI_\Prism]/p^n\cO_\Prism[1/\cI_\Prism].
\end{align*}
Let $\Vect(\cX_\Prism, \cO_\Prism[1/\cI_\Prism]_p^\wedge)^{\varphi = 1}$ denote the category of \textit{Laurent $F$-crystals} on $\cX_\Prism$. This is the category of pairs $(E, \varphi_E)$ consisting of a vector bundle $E$ on the ringed site $(\cX_\Prism, \cO_\Prism[1/\cI_\Prism]_p^\wedge)$ along with an isomorphism $\varphi_E: \varphi^\ast E \xrightarrow{\sim} E$. These admit a similar concrete description as in the case of prismatic $F$-crystals.  Let $\cX_\eta$ be the adic space generic fiber of $\cX$ (with respect to $\bZ_p$), and let $\Loc_{\bZ_p}(\cX_\eta^\lozenge)$ be the category of $\bZ_p$-local systems on $X_\eta^\lozenge$. By Artin-Schreier theory (see \cite[Cor. 3.8]{BS21}), there is a natural equivalence of categories
\begin{align} \label{eq-artinschreier}
	\Vect(\cX_\Prism, \cO_\Prism[1/\cI_\Prism]_p^\wedge)^{\varphi = 1} \xrightarrow{\sim} \Loc_{\bZ_p}(\cX_\eta)
\end{align}
for any bounded $p$-adic formal scheme $\cX$. 

Suppose now $\cX = \Spf(\cO_K)$. Then the right-hand side of (\ref{eq-artinschreier}) identifies with $\Rep_{\bZ_p}(\Gamma_K)$, and the functor (\ref{eq-artinschreier}) is given by 
\begin{align}\label{eq-explicitetalerealization}
	E \mapsto E(W(\cO_C^\flat), \ker(\theta))^{\varphi = 1},
\end{align}
where $C$ is the $p$-adic completion of $\bar{K}$, and the $\Gamma_K$-action on the right-hand side is induced by functoriality and the prismatic crystal property.

Via base change $\cO_\Prism \to \cO_\Prism[1/\cI_\Prism]_p^\wedge$, we obtain a functor 
\begin{align*}
	T: \Vect^\varphi(\Spf(\cO_K)_\Prism, \cO_\Prism) \to \Vect(\Spf(\cO_K)_\Prism, \cO_\Prism[1/\cI_\Prism]_p^\wedge)^{\varphi=1} \xrightarrow{\sim} \Rep_{\bZ_p}(\Gamma_K),
\end{align*}
called the \textit{\'etale realization functor}. By \cite[Prop. 5.3]{BS21}, if $(\cE, \varphi_{\cE})$ is a prismatic $F$-crystal on $\Spf(\cO_K)_\Prism$, then $T(\cE)[1/p]$ is a crystalline $\Gamma_K$-representation. The following is the main theorem of \cite{BS21}.

\begin{thm}[Bhatt-Scholze] \label{thm-bs}
	The \'etale realization functor
	\begin{align*}
		T: \Vect^\varphi(\Spf(\cO_K)_\Prism, \cO_\Prism) \to \Rep_{\bZ_p}^\crys(\Gamma_K)
	\end{align*}
	is an equivalence of tensor categories.
\end{thm}
\begin{proof}
	That $T$ is an equivalence is \cite[Thm. 5.6]{BS21}. Moreover, that $T$ is compatible with tensor products follows from its definition and the description (\ref{eq-explicitetalerealization}) of the equivalence (\ref{eq-artinschreier}). We conclude using the fact that a tensor functor which is an equivalence of categories is necessarily an equivalence of tensor categories \cite[I, 4.4]{SR72}.
\end{proof}

We denote the quasi-inverse tensor functor for $T$ by
\begin{align}\label{eq:U}
	U: \Rep_{\bZ_p}^\crys(\Gamma_K) \xrightarrow{\sim} \Vect^\varphi(\Spf(\cO_K)_\Prism, \cO_\Prism).
\end{align}

\section{Prismatic $F$-crystals and shtukas} \label{section-prismatic}

\subsection{From prismatic $F$-crystals to shtukas}\label{sub-prismtoshtuka}
In this section we describe a process for obtaining shtukas from prismatic $F$-crystals, with the goal of defining a $\cG$-shtuka over $\Spd(\cO_K)$ from a $\cG$-valued crystalline representation of $\Gamma_K$.

Let $\cX$ be a formal scheme which is finite type over $\bZ_p$. Following some ideas from \cite[\textsection 4.4]{PR2021}, we can associate to any prismatic $F$-crystal $(\cE, \varphi_{\cE})$ over $\cX_\Prism$ a shtuka $(\sV, \phi_{\sV})$ over $\cX^\lozenge$. For simplicity, suppose $\cX = \Spf(A)$ is affine. Let $S = \Spa(R,R^+)$ be an affinoid perfectoid space over $k$, and let $(S^\sharp, x) \in \cX^\lozenge(S)$ with $S^\sharp = \Spa(R^\sharp, R^{\sharp+})$. Then $x$ determines a ring homomorphism $A \to R^{\sharp+}$, and therefore a map of formal schemes $\Spf(R^{\sharp+}) \to \Spf(A)$, which we denote also by $x$. Since the map $\theta: W(R^+) \to R^{\sharp+}$ is surjective, and the pair $(W(R^+), \ker(\theta))$ is a bounded prism, we obtain an object 
\begin{align*}
	\tilde{x} = ((W(R^+), \ker(\theta)), x: \Spf(R^{\sharp+}) \to \Spf(A))
\end{align*}
in $\cX_\Prism$. 
	
Evaluating $(\cE, \varphi_\cE)$ on $\tilde{x}$, we obtain a BKF-module $(\cE_{W(R^+)}, \varphi_{\cE_{W(R^+)}})$ on $S$ with one leg at $S^\sharp$. By pulling this back along the morphism of locally ringed spaces
\begin{align*}
	\cY_{[0,\infty)}(S) \to \Spa(W(R^+)) \to \Spec(W(R^+)),
\end{align*}  
this BKF-module induces a shtuka $(\sV_{S}, \phi_{\sV_S})$ over $S$ with one leg at $S^\sharp$. The compatibility isomorphisms coming from the prismatic $F$-crystal induce isomorphisms of the corresponding shtukas, hence the assignment 
\begin{align}\label{eq-crystaltoshtuka}
	(S^\sharp, x) \in \cX^\lozenge(S) \mapsto (\sV_{S}, \phi_{\sV_S}).
\end{align}
determines a shtuka $(\sV,\phi_{\sV})$ over $\cX^\lozenge / \Spd(\bZ_p)$.

\begin{Def} \label{cons-crystaltoshtuka}
	The shtuka $(\sV,\phi_{\sV})$ over $\cX^\lozenge / \Spd(\bZ_p)$ defined by (\ref{eq-crystaltoshtuka}) is the \textit{shtuka associated to the prismatic $F$-crystal $(\cE, \varphi_{\cE})$}.
\end{Def}
	
Let us return to the notation of the previous section: Let $K$ denote a complete discretely valued extension of $\bQ_p$ with perfect residue field $k$, and let $\cO_K$ be its ring of integers. Let $C$ denote the completion of $\bar{K}$. 

We have a functor 
\begin{align}\label{eq-shtukastoLaurent}
	\Sht(\Spd(K)) \to \Vect(\Spf(\cO_K)_\Prism, \cO_\Prism[1/\cI_\Prism]_p^\wedge)^{\varphi=1}
\end{align}
which factorizes the base change
\begin{align*}
	\Vect^\varphi(\Spf(\cO_K)_\Prism, \cO_\Prism) \to \Vect(\Spf(\cO_K)_\Prism, \cO_\Prism[1/\cI_\Prism]_p^\wedge)^{\varphi=1}
\end{align*}
along $\cO_\Prism \to \cO_\Prism[1/\cI_\Prism]_p^\wedge$. To construct this functor, first note that by \cite[Ex. 11.12]{Scholze2017}, there is a natural isomorphism
\begin{align*}
	\Spd(C) \times_{\Spd(K)} \Spd(C) \xrightarrow{\sim} \Spd(C^0(\Gamma_K, C), C^0(\Gamma_K, \cO_C)).
\end{align*}
For brevity, let $\tilde{C} = C^0(\Gamma_K,C)$ and $\tilde{C}^+ = C^0(\Gamma_K, \cO_C)$. Let $p_1, p_2: C^\flat \to \tilde{C}^\flat$ be the two morphisms which induce the projections $\Spd(C) \times_{\Spd(K)} \Spd(C) \to \Spd(C)$. Then $\Vect((\Spf(\cO_K)_\Prism, \cO_\Prism[1/\cI_\Prism]_p^\wedge)^{\varphi=1}$ can be interpreted as the category of $W(C^\flat)$-modules $N$ equipped with a Frobenius $\varphi_N: \varphi^\ast N \xrightarrow{\sim} N$ and an isomorphism
\begin{align}\label{eq-descentdatum}
	\beta_N: N \otimes_{W(C^\flat),p_1} W(\tilde{C}^\flat) \xrightarrow{\sim} N \otimes_{W(C^\flat),p_2} W(\tilde{C}^\flat)
\end{align}
which is compatible with the Frobenius on both sides. For this description see, for example, \cite[Proof of Thm. 5.6]{Wu2021}. 

Now let $\sV$ be a shtuka over $\Spd(K)$. The evaluation of $\sV$ on $\Spa(C, \cO_C)$ can be pulled back along $\Spa(W(C^\flat)) \to \Spa(C^\flat, \cO_C^\flat) \bdtimes \bZ_p$ to obtain a $\varphi$-module $N$ over $W(C^\flat)$. Furthermore, since $\sV$ is defined over $\Spd(K)$, it is equipped with a descent datum $\alpha_\sV:p_1^\ast \sV_{\Spa(C^\flat)} \xrightarrow{\sim} p_2^\ast \sV_{\Spa(C^\flat)}$ over $\Spd(\tilde{C}, \tilde{C}^+)$. By pulling back this descent datum along $W(\tilde{C}^\flat) \to \Spa(\tilde{C}^\flat, \tilde{C}^{\flat+}) \bdtimes \bZ_p$, we obtain an isomorphism $\beta_N$ as in (\ref{eq-descentdatum}).

\begin{lemma}\label{lem-commutes}
	The composition of functors
	\begin{align}\label{eq-comp1}
		\Sht(\Spd(K)) \xrightarrow{(\ref{eq-shtukaslocsys})} \Loc_{\bZ_p}(\Spd(K)) \xrightarrow{(\ref{eq-locsys})} \Loc_{\bZ_p}(\Spa(K))
	\end{align}
	is naturally isomorphic to the composition
	\begin{align}\label{eq-comp2}
		\Sht(\Spd(K)) \xrightarrow{(\ref{eq-shtukastoLaurent})} \Vect(\Spf(\cO_K)_\Prism, \cO_\Prism[1/\cI_\Prism]_p^\wedge)^{\varphi=1} \xrightarrow{(\ref{eq-artinschreier})} \Loc_{\bZ_p}(\Spa(K)).
	\end{align}
%	The following diagram 2-commutes:
%	\begin{equation*}
%		\begin{tikzcd}
%			\Sht(\Spd(K)) 
%				\arrow[r, "(\ref{eq-shtukaslocsys})"] \arrow[d, "(\ref{eq-shtukastoLaurent})"]
%			& \Loc_{\bZ_p}(\Spd(K))
%				\arrow[d, "(\ref{eq-locsys})"]
%			\\ \Vect(\Spf(\cO_K)_\Prism, \cO_\Prism[1/\cI_\Prism]^\wedge_p)^{\varphi=1}
%				\arrow[r, "(\ref{eq-artinschreier})"]
%			& \Loc_{\bZ_p}(\Spa(K)).
%		\end{tikzcd}
%	\end{equation*}
\end{lemma}
\begin{proof}
	Let $\sV$ be a shtuka on $\Spd(K)$. Denote by $\bL$ and $\bL'$ the $\bZ_p$-local systems defined by (\ref{eq-comp1}) and (\ref{eq-comp2}), respectively. By pro-\'etale descent, it is enough to show that $\bL$ and $\bL'$ are isomorphic over $\Spa(C)$, compatibly with their descent data from $\Spa(\tilde{C},\tilde{C}^+)$. 
	
	Following the notation above the statement of the lemma, let $N$ denote the $\varphi$-$W(C^\flat)$-module obtained from pulling back the evaluation of $\sV$ on $\Spa(C,\cO_C)$ to $W(C^\flat)$. Then both $\bL_{\Spa(C)}$ and $\bL'_{\Spa(C)}$ are given by the finite free $\bZ_p$-module $N^{\varphi = 1}.$ Indeed, that $\bL_{\Spa(C)}$ is obtained in this way follows from the fact that the functor from $\phi$-modules over the integral Robba ring $\widetilde{\mathcal{R}}^\mathrm{int}$ to $\bZ_p$-local systems factors through the base extension functor $\widetilde{\mathcal{R}}^\mathrm{int} \to W(C^\flat)$, see \cite[Thm. 8.5.3]{KL} and \cite[Thm. 12.3.4]{SW2020}. That $\bL'_{\Spa(C)}$ is obtained this way follows from the definitions of (\ref{eq-shtukastoLaurent}) and (\ref{eq-artinschreier}), see (\ref{eq-explicitetalerealization}).
	
	It remains to check that the pro-\'etale descent data are the same. First recall the isomorphism $\beta_N$ as in (\ref{eq-descentdatum}), which defines the Laurent $F$-crystal associated to $\sV$. This is an isomorphism in the category $\Vect(W(\tilde{C}^\flat))^{\varphi = 1}$. By Artin-Schreier theory (e.g., see \cite[Prop. 3.2.7]{KL}) and tilting, $\Vect(W(\tilde{C}^\flat))^{\varphi = 1}$ is equivalent to $\Loc_{\bZ_p}(\Spa(\tilde{C}))$, and the descent datum for $\bL'$ is obtained by passing $\beta_N$ through this equivalence.
	
	On the other hand, the descent datum for $\bL$ is obtained by applying the functor $\Sht(\Spd(\tilde{C})) \to \Loc_{\bZ_p}(\Spd(\tilde{C}))$ from (\ref{eq-shtukaslocsys}) to the descent datum of shtukas $\alpha_\sV$ associated to $\sV$. Explicitly, we first obtain from $\alpha_\sV$ an isomorphism of $\phi$-modules over $\widetilde{\mathcal{R}}^\mathrm{int}_{\tilde{C}^\flat}$, which induces an isomorphism of $\bZ_p$-local systems on $\Spa(\tilde{C})$ by \cite[Thm. 8.5.3]{KL} and tilting. But once again the equivalence in \cite[Thm. 8.5.3]{KL} factors through base change along $\widetilde{\mathcal{R}}^\mathrm{int}_{\tilde{C}^\flat} \to W(\tilde{C}^\flat)$, and the isomorphism of $\phi$-modules over $W(C^\flat)$ obtained from $\alpha_{\sV}$ in this way is equal to $\beta_N$. Hence the descent datam for $\bL$ is obtained from $\beta_N$ as well, and the result follows.
\end{proof}

%There is a natural base change functor 
%\begin{align*}
%	\Sht(\Spd(\cO_K)) \to \Sht(\Spd(K)), \ (\sP, \phi_{\sP}) \mapsto (\sP_K, \phi_{\sP_K})
%\end{align*}
%given by restriction: If $(\sP, \phi_{\sP})$ is a shtuka over $\Spd(\cO_K)$, $S$ is in $\Perfk$ with an untilt given by $S^\sharp$, and $f: S^\sharp \to \Spa(K, \cO_K)$ is a morphism of adic spaces, then the value of $(\sP_K, \phi_{\sP_K})$ on $(S^\sharp, f)$ is defined to be the value of $(\sP, \phi_{\sP})$ on $(S^\sharp, S^\sharp \to \Spa(K, \cO_K) \to \Spa(\cO_K))$. 
Consider now the composition of functors
\begin{align}\label{eq-composition}
	\Rep_{\bZ_p}^\crys(\Gamma_K) \to \Sht(\Spd(\cO_K)) \to \Sht(\Spd(K)) \to \Loc_{\bZ_p}(\Spd(K)) \xrightarrow{\sim} \Rep_{\bZ_p} (\Gamma_K).
\end{align}
Here the first arrow is obtained by composing $U$ (see (\ref{eq:U})) with the functor from Definition \ref{cons-crystaltoshtuka}, the second arrow is restriction, the third arrow is induced by (\ref{eq-shtukaslocsys}), and the fourth arrow is the equivalence from Lemma \ref{lem-locsys}.

\begin{lemma} \label{lem-inclusion}
	The composition $(\ref{eq-composition})$ is naturally isomorphic to the inclusion
	\begin{align*}
		\Rep_{\bZ_p}^\crys(\Gamma_K) \hookrightarrow \Rep_{\bZ_p}(\Gamma_K). 
	\end{align*}
\end{lemma}
\begin{proof}
	By Lemma \ref{lem-commutes}, (\ref{eq-composition}) can be rewritten as
	\begin{align}\label{eq-inclusion}
		\Rep_{\bZ_p}^\crys(\Gamma_K) \to \Sht(\Spd(K)) \to \Vect(\Spf(\cO_K)_\Prism,\cO_\Prism[1/\cI_\Prism]_p^\wedge)^{\varphi=1} \xrightarrow{\sim} \Rep_{\bZ_p} (\Gamma_K).
	\end{align}
	In turn, the composition 
	\begin{align}\label{eq-shortcomposition}
		\Vect^\varphi(\Spf(\cO_K)_\Prism,\cO_\Prism) \to \Sht(\Spd(K)) \to \Vect(\Spf(\cO_K)_\Prism,\cO_\Prism[1/\cI_\Prism]_p^\wedge)^{\varphi=1}
	\end{align}
	is isomorphic to base change along $\cO_\Prism \to \cO_\Prism[1/\cI_\Prism]_p^\wedge$. But this base change composed with the final arrow in (\ref{eq-shortcomposition}) defines $T$, so (\ref{eq-composition}) is given by $T\circ U$, which is naturally isomorphic to the identity. 
\end{proof}

\begin{thm}\label{thm-fullyfaithful}
	The functor 
	\begin{align}\label{eq-reptoshtuka}
		\Rep_{\bZ_p}^\crys(\Gamma_K) \to \Sht(\Spd(\cO_K)),
	\end{align}
	defined by composing the equivalence from Theorem \ref{thm-bs} with the functor from Definition \ref{cons-crystaltoshtuka}, is fully faithful.
\end{thm}

\begin{proof}
	Fix a uniformizer $\pi$ of $\cO_K$, and let $(\pi^{1/p^n})_n$ be a compatible system of $p^n$-th roots of $\pi$ inside $\bar{K}$. Write $K_\infty$ for the perfectoid field given by the $p$-adic completion of $\bigcup_n K(\pi^{1/p^n})$. Consider the functor
	\begin{align*}
		\Rep_{\bZ_p}^\crys(\Gamma_K) \to \Sht(\Spd(\cO_K)) \to \Sht(\Spa(K_\infty^\flat, \cO_{K_\infty}^\flat)) \to \Rep_{\bZ_p}(\Gamma_{K_\infty})
	\end{align*}
	given by post-composing (\ref{eq-reptoshtuka}) with evaluation on $K_\infty$ and the functor (\ref{eq-shtukaslocsys}). By Lemma \ref{lem-inclusion}, this composition is isomorphic to the restriction functor $\Rep_{\bZ_p}^\crys(\Gamma_K) \to \Rep_{\bZ_p}(\Gamma_{K_\infty})$, so the composition is fully faithful by \cite[Cor. 2.1.14]{Kisin2006}. Hence it remains only to show the functor
	\begin{align}\label{eq-wantfaithful}
		\Sht(\Spd(\cO_K)) \to \Rep_{\bZ_p}(\Gamma_{K_\infty})
	\end{align}
	is faithful. In turn, it is enough to show the composition of (\ref{eq-wantfaithful}) with the forgetful functor $\Rep_{\bZ_p}(\Gamma_{K_\infty}) \to \Vect(\bZ_p)$ is faithful. But this composition factors as 
	\begin{align}\label{eq-factored}
		\Sht(\Spd(\cO_K)) \to \Sht(\Spd(\cO_C)) \to \Sht(\Spa(C^\flat,\cO_C^\flat)) \to \Vect(\bZ_p),
	\end{align}
	where $C$ is the completion of the algebraic closure $\bar{K}$ of $K$. The first arrow in (\ref{eq-factored}) is faithful since $\Spd(\cO_C) \to \Spd(\cO_K)$ is a $v$-cover (see \cite[Lem. 18.1.2]{SW2020}), the second is faithful by \cite[Thm. 2.7.6]{PR2021}, and the third is faithful by Fargues's Theorem \cite[Thm. 14.1.1]{SW2020}. 
\end{proof}	

\begin{rmk}
	The idea of the proof of Theorem \ref{thm-fullyfaithful} follows the ideas of \cite[Prop. 2.2.17, \href{https://arxiv.org/pdf/2106.08270v2.pdf}{arXiv v2}]{PR2021}. Note, however, that the construction given in Example 2.2.16 of \textit{loc. cit.} does not work as written, and the proposition was removed from subsequent versions.
\end{rmk}

\begin{rmk}
	It is not hard to see that the functor (\ref{eq-reptoshtuka}) is faithful, since $U$ is fully faithful and the functor from Definition \ref{cons-crystaltoshtuka} is essentially given by restriction to perfectoid objects. The difficulty lies in proving it is full. Indeed, a priori one needs descent data to recover the morphism of prismatic $F$-crystals from the morphism of shtukas, but $\cO_C \widehat{\otimes}_{\cO_K} \cO_C$ is only quasiregular semiperfectoid, not perfectoid.
\end{rmk}

\begin{cor}\label{cor-fullyfaithful}
	The functor 
	\begin{align*}
		\Vect^\varphi((\Spd(\cO_K))_\Prism, \cO_\Prism) \to \Sht(\Spd(\cO_K))
	\end{align*} 
	from Definition \ref{cons-crystaltoshtuka} is fully faithful. \hfill \qed
\end{cor}
%\begin{proof}
%	Combine Theorem \ref{thm-bs} with Theorem \ref{thm-fullyfaithful}.
%\end{proof}

\subsection{$\cG(\bZ_p)$-valued crystalline representations and $\cG$-shtukas}
In this section we define a Tannakian variant of the construction from the previous section. Let $K$ and $\cO_K$ be as in the previous section, and assume moreover that the residue field $k$ of $K$ is either finite or algebraically closed. Let $\cG$ be a smooth group scheme over $\bZ_p$ with connected fibers and generic fiber $G = \cG\otimes_{\bZ_p} \bQ_p$. 
\begin{Def}
	A \textit{$\cG$-valued crystalline representation of $\Gamma_K$} is an exact $\bZ_p$-linear tensor functor 
	\begin{align*}
		\alpha: \Rep_{\bZ_p}(\cG) \to \Rep_{\bZ_p}^\crys(\Gamma_K).
	\end{align*}
\end{Def}

We denote the category of $\cG$-valued crystalline representations of $\Gamma_K$ by $\cG$-$\Rep^\crys_{\bZ_p}(\Gamma_K)$.

\begin{rmk}\label{rmk-forget}
	Let $\alpha$ be a $\cG$-valued crystalline representation of $\Gamma_K$. Then by Lemma \ref{lem-tannakiantorsors}, the composition
	\begin{align*}
		\Rep_{\bZ_p}(\cG) \xrightarrow{\alpha} \Rep_{\bZ_p}^\crys(\Gamma_K) \hookrightarrow \Rep_{\bZ_p}(\Gamma_K) \xrightarrow{\sim} \Loc_{\bZ_p}(\Spd K)
	\end{align*}
	is pro-\'etale locally trivial. It follows from this and Remark \ref{rmk-forgetful} that the composition of $\alpha$ with the forgetful functor $\Rep_{\bZ_p}^\crys(\Gamma_K) \to \Vect(\bZ_p)$ yields the forgetful functor $\Rep_{\bZ_p}(\cG) \to \Vect(\bZ_p)$.
\end{rmk}

 %Following \cite[\textsection 4.4.4]{KSZ}, we say $\cG$ \textit{satisfies property KL} if all $\cG$-torsors on $U$ are trivial. In particular, by \cite[Cor. 1.2]{Anschutz22}, all parahoric group schemes $\cG$ satisfy KL. 

%Let $\cG$ be a flat affine group scheme of finite type over $\bZ_p$, and let $\Rep_{\bZ_p}(\cG)$ denote the category of algebraic representations of $\cG$ on finite free $\bZ_p$-modules. To any $\cG(\bZ_p)$-valued representation of $\Gamma_K$ we can associate an exact $\bZ_p$-linear tensor functor
%\begin{align*}
%	\alpha: \Rep_{\bZ_p}(\cG) \to \Rep_{\bZ_p}(\Gamma_K)
%\end{align*}
%by assigning to any representation $(\Lambda, \rho)$ of $\cG$ the $\Gamma_K$-representation obtained by the composition $\rho \circ \alpha$.

The following is the main result in this section. Let $D^\times = \Spec(W(k))\llbracket u \rrbracket \setminus \{(p,u)\}$. In the following proposition we will assume that all $\cG$-torsors over $D^\times$ are trivial. This is the case when $\cG$ is parahoric by Ansch\"utz's Theorem, \cite[Cor 1.2]{Anschutz22}.

\begin{prop}\label{prop-exacttensor}
	Suppose $\cG$ is a smooth group scheme over $\bZ_p$ with connected fibers and with the property that every $\cG$-torsor over $D^\times$ is trivial. Then for every $\cG$-valued crystalline representation $\alpha$ of $\Gamma_K$, the functor
	\begin{align*}
		U\circ \alpha: \Rep_{\bZ_p}(\cG) \to \Vect^\varphi(\Spf(\cO_K)_\Prism, \cO_\Prism)
	\end{align*}
	is an exact $\bZ_p$-linear tensor functor.
\end{prop}
\begin{proof}
	Let $\alpha$ be a $\cG$-valued crystalline representation of $\Gamma_K$, and suppose $(A,I)$ is a prism over $\Spf(\cO_K)$. Define $\omega^\alpha_A$ to be the composition of functors
	\begin{align}\label{eq-omegaA}
		\omega^\alpha_A: \Rep_{\bZ_p}(\cG) \xrightarrow{\alpha} \Rep_{\bZ_p}^\crys(\Gamma_K) \xrightarrow{U} \Vect^\varphi(\Spf(\cO_K)_\Prism, \cO_\Prism) \to \Vect(A),
	\end{align}
	where the final arrow is given by the composition of evaluation on $((A,I),\Spf(A/I) \to \Spf(\cO_K))$ and the forgetful functor $\Vect^\varphi(A) \to \Vect(A)$. It is enough to prove that $\omega^\alpha_A$ is an exact $\bZ_p$-linear tensor functor for every prism $(A,I)$ over $\Spf(\cO_K)$.
	
	By Theorem \ref{thm-bs}, to show that $\omega^\alpha_A$ is a tensor functor, we need only show that evaluation $\Vect(\Spf(\cO_K)_\Prism, \cO_\Prism) \to \Vect(A)$ is a tensor functor. In general, if $\cE_1$ and $\cE_2$ are two prismatic crystals, then their tensor product $\cE_1 \otimes_{\cO_\Prism} \cE_2$ is the sheaf associated to the presheaf $(B,J) \mapsto \cE_1(B) \otimes_B \cE_2(B)$. But this is already a sheaf for the faithfully flat topology on prisms by the crystal properties of $\cE_1$ and $\cE_2$ combined with the sheaf property of $\cO_\Prism$ and the flatness of $\cE_i(B)$.
	
	It remains to show that $\omega^\alpha_A$ is exact. We first show this for the Breuil-Kisin prism. Let $\fS = W(k)\llbracket u \rrbracket$. By choosing a uniformizer $\pi \in \cO_K$ we obtain a surjection $\theta_K: \fS \to \cO_K$ by sending $u$ to $\pi$, and the kernel of this map is generated by an Eisenstein polynomial $E(u)$. The ring $\fS$ becomes a $\delta$-ring in the sense of \cite{BS22} when we extend the Witt vector Frobenius to $\fS$ by $u \mapsto u^p$, and $(\fS, (E(u)))$ defines a bounded prism (called the \textit{Breuil-Kisin prism} in \cite{BS21}). Moreover, the surjection $\theta_K$ endows $(\fS, (E(u)))$ with the structure of prism over $\Spf(\cO_K)$. We claim the tensor functor $\omega^\alpha_\fS: \Rep_{\bZ_p}(\cG) \to \Vect(\fS)$ is exact.
	
	By \cite[Rmk. 7.11]{BS21}, the functor 
	\begin{align*}
		\Rep_{\bZ_p}^\crys(\Gamma_K) \to \Vect^\varphi(\Spf(\cO_K)_\Prism, \cO_\Prism) \to \Vect^\varphi(\fS)
	\end{align*}
	is naturally isomorphic to the functor $\fM$ defined by Kisin \cite[Thm. 1.2.1]{Kisin2010}. Since all $\cG$-torsors on $D^\times$ are trivial, the composition
	\begin{align*}
		\Rep_{\bZ_p}(\cG) \xrightarrow{\alpha} \Rep_{\bZ_p}^\crys(\Gamma_K) \xrightarrow{\fM} \Vect^\varphi(\fS) \to \Vect(\fS),
	\end{align*}
	which is isomorphic to $\omega^\alpha_{\fS}$, is exact by \cite[Lem. 4.4.5]{KSZ} (note that the proof in \textit{loc. cit.} is in the case where $k$ is finite, but the same argument goes through when $k$ is algebraically closed). Therefore $\omega^\alpha_{\fS}$ is exact.
	
	Now let $(A,I)$ be an arbitrary bounded prism over $\Spf(\cO_K)$. By \cite[Ex. 2.6 (1)]{BS21}, there exists a faithfully flat map of prisms $(A,I) \to (B,IB)$ for which there is a map $(\fS, (E(u))) \to (B, IB)$ in $\Spf(\cO_K)_\Prism$. By the crystal property, the functor $\omega^\alpha_B$ is naturally isomorphic to the functor
	\begin{align*}
		\Rep_{\bZ_p}(\cG)  \to \Vect(B), \ (\Lambda,\rho)\mapsto \omega^\alpha_{\fS}(\Lambda) \otimes_{\fS} B,
	\end{align*}
	and therefore $\omega^\alpha_B$ is exact as well. Moreover, the crystal property applied again implies that $\omega^\alpha_B$ is naturally isomorphic to the functor $(\Lambda, \rho) \mapsto \omega^\alpha_A(\Lambda) \otimes_A B$, and therefore we can conclude that $\omega_A^\alpha$ is exact by adically flat descent (see e.g., \cite[Prop. 6.1.11, Chapter I]{FujiwaraKato}).
	
%	Now suppose $0 \to \Lambda' \to \Lambda \to \Lambda'' \to 0$ is exact in $\Rep_{\bZ_p}(\cG)$. Then the sequence
%	\begin{align*}
%		0 \to \omega^\alpha_A(\Lambda') \otimes_A B \to \omega^\alpha_A(\Lambda) \otimes_A B \to \omega^\alpha_A(\Lambda'')\otimes_A B \to 0
%	\end{align*}
%	is exact in $\Vect(B)$. It follows that the sequence
%	\begin{align*}
%		0 \to \omega^\alpha_A(\Lambda')\otimes_A B/(p,I)^nB \to \omega^\alpha_A(\Lambda) \otimes_A B/(p,I)^nB \to \omega^\alpha_A(\Lambda'') \otimes_A B/(p,I)^nB \to 0
%	\end{align*}
%	is exact for all $n$. By $(p,I)$-faithful flatness of $(A,I) \to (B,IB)$, the maps $A/(p,I)^n \to B/(p,I)^nB$ are faithfully flat for all $n$ (see Remark \ref{rmk-ff}), so the sequence
%	\begin{align*}
%		0 \to \omega^\alpha_A(\Lambda')\otimes_A A/(p,I)^n \to \omega^\alpha_A(\Lambda) \otimes_A A/(p,I)^n \to \omega^\alpha_A(\Lambda'') \otimes_A A/(p,I)^n \to 0
%	\end{align*}
%	is exact for all $n$. But $\omega^\alpha_A(\Lambda)$ is finite projective for all $\Lambda$, and therefore classically $(p,I)$-complete as well (since $A$ is). Hence we can pass to the limit to obtain the desired exactness.
\end{proof}

For any tensor functors $\omega: \Rep_{\bZ_p} (\cG) \to \Vect(A)$, denote by $\omega \otimes_A B$ the functor given by post-composition with the base change $\Vect(A) \to \Vect(B)$. If $\omega_1$ and $\omega_2$ are two tensor functors $\Rep_{\bZ_p}(\cG) \to \Vect(A)$, let $\ul{\Isom}^\otimes(\omega_1, \omega_2)$ be the fpqc sheaf on $\Spec(A)$ which assigns to any $A$-algebra $B$ the set of isomorphisms of tensor functors $\omega_1 \otimes_A B \xrightarrow{\sim} \omega_1 \otimes_A B$. Write $\ul{\Aut}^\otimes(\omega)$ for $\ul{\Isom}^\otimes(\omega,\omega)$. 

For any $\bZ_p$-algebra $A$, define the standard fiber functor by
\begin{align*}
	\mathbbm{1}_A: \Rep_{\bZ_p} (\cG) \to \Vect(A), \ (\Lambda, \rho) \mapsto \Lambda \otimes_{\bZ_p} A.
\end{align*}
By the reconstruction theorem in Tannakian duality (see \cite[Thm. 5.17]{Wedhorn2004} for the statement in this situation), the canonical morphism $\cG \to \ul{\Aut}^\otimes(\mathbbm{1}_{\bZ_p})$ is an isomorphism of fpqc sheaves. As a result, $\ul{\Isom}^\otimes(\mathbbm{1}_A, \omega)$ carries a natural right $\cG_A$-action for any tensor functor $\omega: \Rep_{\bZ_p}(\cG) \to \Vect(A)$. It follows from the Tannakian formalism \cite[Thm. 19.5.1]{SW2020} and Proposition \ref{prop-exacttensor} that when $\omega= \omega^\alpha_A$ is defined as in (\ref{eq-omegaA}), $\ul{\Isom}^\otimes(\mathbbm{1},\omega^\alpha_A)$ is a $\cG$-torsor over $\Spec(A)$, which we denote by $\sP^\alpha_A$.

The Frobenius morphisms for $(\alpha \circ U)(\Lambda, \rho)$ as $(\Lambda, \rho)$ varies piece together to define an isomorphism of $\cG$-torsors
\begin{align*}
	\phi_{\sP^\alpha_A}: \varphi_A^\ast (\sP^\alpha_A) \res_{\Spec(A) \setminus V(I)} \xrightarrow{\sim} \sP^\alpha_A\res_{\Spec(A) \setminus V(I)}.
\end{align*}
over $\Spec(A) \setminus V(I)$. In particular, if $S = \Spa(R,R^+)$ is an affinoid perfectoid space over $k$, with untilt $S^\sharp = \Spa(R^\sharp, R^{\sharp+})$ admitting a homomorphism $\cO_K \to R^{\sharp+}$, this construction determines a $\cG$-BKF-module $(\sP^\alpha_{W(R^+)}, \phi_{\sP^\alpha_{W(R^+)}})$ over $S$ with one leg at $S^\sharp$. Pulling back along 
\begin{align*}
	S \bdtimes \bZ_p \to \Spa(W(R^+)) \to \Spec(W(R^+)),
\end{align*}
we obtain a $\cG$-shtuka over $S$ with one leg at $S^\sharp$, which we denote by $(\sP^\alpha_S, \phi_{\sP^\alpha_S})$. It follows that the assignment
\begin{align}\label{eq-reptoGshtuka}
	(S^\sharp, f) \in (\Spd (\cO_K))(S) \mapsto (\sP^\alpha_S, \phi_{\sP^\alpha_S})
\end{align}
defines a shtuka over $\Spd(\cO_K)$.

\begin{Def}\label{def-greptoshtuka}
	Let $\alpha$ be a $\cG$-valued crystalline representation of $\Gamma_K$. The shtuka $(\sP^\alpha, \phi_{\sP^\alpha})$ over $\Spd(\cO_K)$ given by the assignment (\ref{eq-reptoGshtuka}) is the \textit{$\cG$-shtuka associated $\alpha$}.
\end{Def}

The assignment $\alpha \mapsto (\sP^\alpha, \phi_{\sP^\alpha})$ satisfies the following obvious functorialities.
\begin{lemma}\label{lem-functorialities}
	Let $\cG$ be a parahoric group scheme over $\bZ_p$, and let $\alpha$ be a $\cG$-valued crystalline representation of $\Gamma_K$.
	\begin{enumerate}[$(a)$]
		\item Suppose $\rho: \cG \to \cG'$ is a morphism of parahoric group schemes over $\bZ_p$. Then 
		\begin{align*}
			\sP^{\rho^\ast(\alpha)} \cong \sP^\alpha \times^{\cG} \cG',
		\end{align*}
		where $\rho^\ast(\alpha)$ is the $\cG'$-valued representation of $\cG'$ induced via $\rho$ by $\alpha$.
		\item Let $K'$ be an extension of $K$ inside of $\bar{K}$, and let $\alpha'$ be the restriction of $\alpha$ to $\Gamma_{K'}$. Then 
		\begin{align*}
			\sP^\alpha \res_{\Spd(\cO_{K'})} \cong \sP^{\alpha'}.
		\end{align*}
	\end{enumerate}
\end{lemma}
\begin{proof}
	The proof is straightforward and left to the reader.
\end{proof}

We close this section with the group theoretic analog of Theorem \ref{thm-fullyfaithful}.

\begin{thm}\label{technicalthm}
	The functor
	\begin{align*}
		\cG\textup{-}\Rep_{\bZ_p}^\crys(\Gamma_K) \to \Sht_\cG(\Spd(\cO_K))
	\end{align*}
	from Definition \ref{def-greptoshtuka} is fully faithful.
\end{thm}
\begin{proof}
	This follows from Theorem \ref{thm-fullyfaithful} and the Tannakian formalism. 
\end{proof}

\subsection{The generic fiber}
In the remainder of this section we show that the construction in Definition \ref{def-greptoshtuka} is compatible with those of \cite{PR2021} over the generic fiber. To be more specific, suppose 
\begin{align*}
	\alpha: \Rep_{\bZ_p}(\cG) \to \Rep_{\bZ_p}^\crys(\Gamma_K)
\end{align*}
is a $\cG$-valued crystalline representation of $\Gamma_K$. Then to $\alpha$ we can attach two $\cG$-shtukas over $\Spd(K)$. We have, on the one hand, the $\cG$-shtuka $(\sP^\alpha, \phi_{\sP^\alpha})$ obtained by pulling back the $\cG$-shtuka from Definition \ref{def-greptoshtuka} along $\Spd(K) \to \Spd(\cO_K)$. On the other hand, the representation $\alpha(\Lambda, \rho)\otimes_{\bZ_p} \bQ_p$ is crystalline for every $(\Lambda, \rho)$, and therefore also de Rham. Hence we can attach to $\alpha$ a $\cG$-shtuka $(\sP^\alpha_{\dR}, \phi_{\sP^\alpha_{\dR}})$ over $\Spd(K)$ using \cite[Def. 2.6.6]{PR2021}. 

\begin{lemma} \label{lem-genericfiber}
	The restriction of the $\cG$-shtuka $(\sP^\alpha, \phi_{\sP^\alpha})$ along $\Spd(K) \to \Spd(\cO_K)$ is given by $(\sP^\alpha_{\dR}, \phi_{\sP^\alpha_{\dR}})$.
\end{lemma}
\begin{proof}
	By \cite[Prop. 2.5.1]{PR2021}, it's enough to show that the two $\cG$-shtukas determine the same pair $(\bP, D)$, where $\bP$ is a pro-\'etale $\ul{\cG(\bZ_p)}$-torsor on $\Spd(K)$ and $D$ is a $\ul{\cG(\bZ_p)}$-equivariant morphism of $v$-sheaves $D: \bP \to \Gr_{G,\Spd(K)}$ over $\Spd(K)$. Here, $\Gr_{G,\Spd(K)}$ denotes the Beilinson-Drinfeld Grassmannian over $\Spd(K)$, see \cite[Def. 20.2.1]{SW2020}.
	
	Let us first address the pro-\'etale $\ul{\cG(\bZ_p)}$-torsors. By Lemma \ref{lem-tannakiantorsors}, it's enough to show the two determine the same tensor functor $\Rep_{\bZ_p}(\cG) \to \Loc_{\bZ_p}(\Spd(K)).$	We claim that for both $\sP^\alpha$ and $\sP_\dR^\alpha$, this functor is given by assigning to $(\Lambda, \rho) \in \Rep_{\bZ_p}(\cG)$ and $(S^\sharp, f)\in \Spd(K)(S)$ the local system obtained by applying the tilting equivalence to the pullback of $\alpha(\Lambda,\rho)$ along $f: S^\sharp \to \Spa(K)$. For $\sP^\alpha_\dR$ this follows from \cite[Def. 2.6.6]{PR2021}.
	
	Let $\bP^\alpha$ denote the pro-\'etale $\ul{\cG(\bZ_p)}$-torsor corresponding to $\sP^\alpha$. By the definition of (\ref{eq-Gshtukastors}), for every $(S^\sharp, f) \in \Spd(K)(S)$, the tensor functor corresponding to $\bP^\alpha_S$ factors as
	\begin{align*}
		\Rep_{\bZ_p}(\cG) \to (\textup{shtukas over $S$ with one leg at $S^\sharp$}) \xrightarrow{(\ref{eq-shtukaslocsys})} \Loc_{\bZ_p}(S),
	\end{align*}
	where here the first arrow assigns to $(\Lambda, \rho)$ the vector bundle shtuka $\sV^\alpha_{\rho}$ over $S$ with one leg at $S^\sharp$ obtained by pushing forward $\sP^\alpha_S$ along $\rho$. By definition of $\sP^\alpha$, the shtuka $\sV^\alpha_{S,\rho}$ is the pullback to $\Spd(K)$ of the shtuka associated to the prismatic $F$-crystal $U(\alpha(\Lambda, \rho))$. Hence it follows from Lemma \ref{lem-inclusion} that the $\bZ_p$-local system on $S$ corresponding to $\sV^\alpha_{S,\rho}$ is obtained by tilting the pullback of $\alpha(\Lambda, \rho)$ to $S^\sharp$ along $f$, as desired. 
	
	It remains to show that the two morphisms $D: \bP \to \Gr_{G,\Spd(K)}$ agree. Let $C$ be the completion of $\bar{K}$. Since $\Spd(C) \to \Spd(K)$ is a pro-\'etale cover, it's enough to show the two morphisms agree after base change to $\Spd(C)$. Moreover, we can reduce to the case $\cG= \GL_n$ by the Tannakian formalism. In this case pairs $(\bP, D)$ correspond to pairs $(T_0,\Xi)$, where $T_0$ is a finite free $\bZ_p$-module and $\Xi$ is a $B_\dR^+(C)$-lattice in $T_0 \otimes_{\bZ_p} B_\dR(C)$. 
	
	Suppose $T$ is a crystalline representation of $\Gamma_K$ on a finite free $\bZ_p$-module with corresponding shtukas $(\sV, \phi_\sV)$ coming from Definition \ref{cons-crystaltoshtuka} and $(\sV_\dR, \phi_{\sV_\dR})$ coming from \cite[Def. 2.6.4]{PR2021}. By the first part of the proof, we know both $(\sV,\phi_{\sV})$ and $(\sV_\dR, \phi_{\sV_\dR})$ correspond to the finite free $\bZ_p$-module underlying $T$. Moreover, the corresponding $B_\dR^+(C)$-lattice $\Xi$ is given by $D_\dR(T[1/p]) \otimes_K B_\dR^+(C)$	in both cases. Indeed, for $(\sV_\dR, \phi_{\sV_\dR})$ this follows from \cite[Prop. 2.6.3]{PR2021}. On the other hand, let $(M, \varphi_M)$ denote the BKF-module coming from the evaluation of $U(T)$ on $(W(\cO_C^\flat), \ker(\theta))$. Then since $(\sV,\phi_\sV)$ is the shtuka associated to $U(T)$ as in Definition \ref{cons-crystaltoshtuka}, it follows from the proof of \cite[Thm. 5.2]{BS21} that the $B_\dR^+(C)$-lattice associated to $(\sV, \phi_\sV)$ is $\phi^\ast M \otimes_{W(\cO_C^\flat)} B_\dR^+(C)$. But by \cite[Rem. 5.4]{BS21}, this agrees with $D_\dR(T[1/p]) \otimes_K B_\dR^+(C)$. 
\end{proof}

Suppose $\alpha$ is a $\cG$-valued crystalline representation of $\Gamma_K$. For any representation $(V,\rho)$ of the generic fiber $G$ of $\cG$, the representation $\alpha(V,\rho)\otimes_{\bZ_p}\bQ_p$ of $\Gamma_K$ is crystalline, therefore Hodge-Tate. The functor assigning to any $(V,\rho)$ the grading on $D_{\textup{HT}}(V)$ defines a graded fiber functor on $\Rep_{\bQ_p}(G_{\bar{K}})$ in the sense of \cite[Def. 3.1]{Ziegler2015}, and therefore it follows from the Tannakian formalism (see e.g., \cite[Cons. 3.4]{Ziegler2015}) that the grading on $D_\mathrm{HT}(V)$ as $V$ varies is determined by a cocharacter $\mu$ of $G_{\bar{K}}$, called the \textit{Hodge-Tate cocharacter} for $\alpha$.

\begin{lemma}\label{lem-bdbymu}
	Let $\alpha$ be a $\cG$-valued crystalline representation of $\Gamma_K$ with Hodge-Tate cocharacter $\mu$. Then the shtuka $(\sP^\alpha, \phi_{\sP^\alpha})$ over $\Spd(\cO_K)$ is bounded by $\mu$.
\end{lemma}
\begin{proof}
	The $\cG$-shtuka $\sP^\alpha_\dR$ is bounded by $\mu$ by \cite[Def. 2.6.6]{PR2021}, so Lemma \ref{lem-genericfiber} implies that $\sP^\alpha_{\Spd(K)}$ is bounded by $\mu$. Since $\Spd(\cO_K)$ is topologically flat by \cite[Lem. 2.17]{AGLR}, the result follows from Lemma \ref{lem-topflat}.
\end{proof}

\section{Shimura varieties of toral type} \label{section-SV}

\subsection{Shimura varieties and crystalline representations}\label{sub-SV}
In this section we recall some definitions and notation from the theory of (global) Shimura varieties. Let $(\eG, X)$ be a Shimura datum, meaning that $\eG$ is a connected reductive group over $\bQ$, and $X$ is a $\eG(\bR)$-conjugacy class of homomorphisms
\begin{align*}
	h:\mathbb{S} = \textup{Res}_{\bC/\bR} \mathbb{G}_m \to \eG_\bR
\end{align*}
satisfying the axioms (SV1)-(SV3) in \cite[Def. 5.5]{Milne2005}. Associated to any $h \in X$ is a cocharacter $\mu_h$ of $\eG_\bC$, whose conjugacy class is defined over a number field $\eE$, called the reflex field of $(\eG, X)$. Denote by $\{\mu\}$ the $G(\overbar{\bQ})$-conjugacy class of any $\mu_h$ for $h \in X$. We will also refer to any choice of an element of $\{\mu\}$ as $\mu$.

Let $\bA_f$ denote the finite adeles over $\bQ$ and $\bA_f^p$ denote the finite adeles away from $p$. For any compact open subgroup $\eK \subset \eG(\bA_f^p)$, we can attach to $(\eG, X)$ and $\eK$ the Shimura variety $\Sh_\eK(\eG,X)$, which is a quasi-projective variety over $\bC$ whose $\bC$-points are given by 
\begin{align*}
	\Sh_\eK(\eG, X)(\bC) = \eG(\bQ) \backslash X \times \eG(\bA_f) / \eK.
\end{align*}
There exists a canonical model for $\Sh_\eK(\eG,X)$ over the reflex field $\eE$, which we will denote by $\Sh_\eK(\eG,X)_\eE$. 

Let $\eZ$ denote the center of $\eG$, and let $\eZ^\circ$ be the identity component of $\eZ$. We write $\eZ_{ac}$ denote the anti-cuspidal part of $\eZ^\circ$ in the sense of \cite[Def. 1.5.4]{KSZ}. Then $\eZ^\circ / \eZ_{ac}$ is cuspidal in the sense of \textit{loc. cit.}, i.e., has equal $\bQ$-split rank and $\bR$-split rank. Denote by $\eG^c$ the quotient $\eG/\eZ_{ac}$, and for any subgroup $\eH \subset \eG(\bA_f)$, denote by $\eH^c$ the image of $\eH$ under $\eG \to \eG^c$. 

%In addition to the axioms above, we assume that $\eZ^\circ$ splits over a CM field (this is (SV6) in \cite{Milne2005}). Denote by $\eZ_{s}$ the maximal subtorus of $\eZ^\circ$ which is $\bQ$-anisotropic and $\bR$-split, and let $\eG^c$ denote the quotient $\eG/\eZ_{s}$. \edit{this condition should be unnecessary}

%For any subgroup $\eH$ of $\eG(\bA_f)$, we denote by $\eH^c$ the image of $\eH$ in $\eG^c(\bA_f)$. In particular, $\eK^c = \eK_p^c \eK^{p,c}$ for compact open subgroups $\eK_p \subset \eG(\bQ_p)$ and $\eK^{p,c} \subset \eG(\bA_f^p)$.
%We fix a smooth integral model $\cG^c$\edit{This needs to be changed - $\mathscr{G}^c$ should be the parahoric subgroup of $G^c$ corresponding to $\mathscr{G}$} of $G^c$ such that $\eG^c(\bZ_p) = \eK_p^c$.

Fix a prime $p$, and let $G = \eG\otimes_\bQ \bQ_p$ and $G^c = \eG^c \otimes_\bQ \bQ_p$ . Assume $\eK = \eK_p\eK^p \subset \eG(\bA_f)$, where $\eK_p$ and $\eK^p$ are compact open subgroups of $\eG(\bQ_p)$ and $\eG(\bA_f^p)$, respectively. We will henceforth assume $\eK^p$ is neat. 

Suppose $\eK_p$ is a parahoric subgroup of $G(\bQ_p)$. Then $\eK_p$ is the connected stabilizer of a point $x$ in the extended Bruhat-Tits building $\mathscr{B}(G,\bQ_p)$ for $G$ over $\bQ_p$. We will denote by $\cG$ the corresponding parahoric group scheme defined in \cite{BTII}, so $\cG$ is a smooth affine group scheme over $\bZ_p$ which satisfies $\cG(\bZ_p) = \eK_p$ and $\cG_{\bQ_p} = G$. Since $G \to G^c$ is a central extension, by \cite[Thm. 2.1.8]{Landvogt2000} we obtain a canonical $G(\bQ_p)$-equivariant map of extended Bruhat-Tits buildings $\mathscr{B}(G,\bQ_p) \to \mathscr{B}(G^c, \bQ_p)$. Let $x^c$ denote the image of $x$ under this map, and denote by $\cG^c$ the parahoric group scheme corresponding to $x^c$. By \cite[1.7.6]{BTII}, the homomorphism $G \to G^c$ extends to a homomorphism of $\bZ_p$-group schemes $\cG \to \cG^c$, see for comparison a related discussion in \cite[1.1.3]{KP2018}. On $\bZ_p$-points, $\cG(\bZ_p) \to \cG^c(\bZ_p)$ factors as a composition
\begin{align}\label{composition}
	\cG(\bZ_p) = \eK_p \twoheadrightarrow \eK_p^c \hookrightarrow \cG^c(\bZ_p).
\end{align}

\begin{rmk}
	A more concrete description of $\cG^c$ is available in some cases. If $Z_{ac} = \ker(G \to G^c)$ is $R$-smooth in the sense of \cite[Def. 2.4.3]{KZ21}, then $\cG^c$ is the quotient of $\cG$ by the Zariski closure $\mathcal{Z}_{ac}$ of $Z_{ac}$ inside of $\cG$ by \cite[Prop. 2.4.14]{KZ21}. This happens in particular if $G$ splits over a tamely ramified extension, see \cite[2.4.5]{KZ21} (cf. \cite[Prop. 1.1.4]{KP2018}). 
\end{rmk}

In the remainder of this section, we focus on the case where $\eG = \eT$ is a $\bQ$-torus. Then the unique parahoric subgroup $\eK_p$ of $T(\bQ_p)$ corresponds to the identity component $\cT$ of the N\'eron model of $T$, and the conjugacy class $X$ reduces to a single homomorphism $h: \mathbb{S} \to \eT_\bR$. For any neat compact open subgroup $\eK \subset \eT(\bA_f)$, we obtain a zero-dimensional Shimura variety, which is given by a finite set of points:
\begin{align} \label{eq-SV}
	\Sh_\eK(\eT,\{h\}) = \coprod_{i\in I} \Spec{\eE_i}, 
\end{align}
where each $\eE_i$ is a finite extension of $\eE$.

We will define below certain $\cT^c(\bZ_p)$-valued crystalline representations of $\Gal(\bar{E}/E)$, which will be used to define shtukas over an integral model for $\Sh_\eK(\eT,\{h\})$. Since $\eT$ is a torus, $\{\mu\}$ is a singleton set, and therefore $\mu$ is defined over $\eE$. Denote by $r(\mu)^\textup{alg}$ the following homomorphism of algebraic groups over $\bQ$:
\begin{align}\label{eq-reflexnorm}
	\Res_{\eE/\bQ} \bG_m \xrightarrow{\Res_{\eE/\bQ} \mu} \Res_{\eE/\bQ} \eT \xrightarrow{\textup{Nm}_{\eE/\bQ}} \eT.
\end{align}
Applying $r(\mu)^\textup{alg}$ to $\bA$-points, we obtain a homomorphism $\bA_\eE^\times \to \eT(\bA)$ which, by functoriality, sends $\eE^\times = (\Res_{\eE/\bQ} \bG_m)(\bQ)$ into $\eT(\bQ)$. Hence, $r(\mu)^\textup{alg}$ induces a homomorphism
\begin{align*}
	\eE^\times \backslash \bA_\eE^\times \to \eT(\bQ) \backslash \eT(\bA).
\end{align*}
For each open compact subgroup $\eK$ of $\eT(\bA_f)$, the quotient $\eT(\bQ) \backslash \eT(\bA_f) / \eK$ is finite, so the natural map $\eT(\bQ) \backslash \eT(\bA) \to \eT(\bQ) \backslash \eT(\bA_f) / \eK$ factors through $\pi_0(\eT(\bQ)\backslash \eT(\bA))$, and the composition $\eE^\times \backslash \bA_\eE^\times \to \eT(\bQ) \bs \eT(\bA_f)/\eK$ factors through $\pi_0(\eE^\times \bs \bA_\eE^\times) \cong \Gal(\eE^\ab / \eE)$ via the global Artin homomorphism, $\textup{Art}_\eE$ (which we normalize geometrically):
\begin{align*}
	\eE^\times \bs \bA_\eE^\times \xrightarrow{\textup{Art}_\eE} \Gal(\eE^\ab / \eE) \to \eT(\bQ) \bs \eT(\bA_f) / \eK.
\end{align*}
We will denote the resulting factorization by $r(\mu)_\eK: \Gal(\eE^\ab / \eE) \to \eT(\bQ) \bs \eT(\bA_f) / \eK$. 

We can now pin down the fields $\eE_i$ appearing in (\ref{eq-SV}) more precisely. Let $\eE_\eK$ be the finite extension of $\eE$ with the property that 
\begin{align}\label{eq-EU}
	\Gal(\eE^\ab / \eE_\eK) = \ker(r(\mu)_\eK).
\end{align}

\begin{lemma}\label{lem-E_K}
	Each of the fields $\eE_i$ is isomorphic to $\eE_\eK$.
\end{lemma}
\begin{proof}	
This follows by recalling the definition of the canonical model for $\Sh_\eK(\eT,\{h\})$ over $\eE$. For $\sigma \in \Gal(\bar{\eE} / \eE)$ and $x \in \Sh_\eK(\eT,\{h\})(\bar{E}) \cong \eT(\bQ) \bs \eT(\bA_f) / \eK$, one uses $r(\mu)_\eK$ to define the reciprocity law
\begin{align}\label{eq-reciprocity}
	\sigma(x) = r(\mu)_\eK(\sigma) \cdot x,
\end{align}
which determines a continuous action of $\Gal(\bar{E} / E)$ on the finite set $\Sh_\eK(\eT,\{h\})$, and therefore the structure of a finite \'etale $E$-scheme on $\Sh_\eK(\eT,\{h\})$. The isomorphisms $\eE_i \cong \eE_\eK$ then follow immediately from the definition of the action (\ref{eq-reciprocity}).
\end{proof}

Passing to the limit along open compact subsets $\eK$ of $\eT(\bA_f)$, we obtain from the collection $\{r(\mu)_\eK\}$ a map
\begin{align*}
	r(\mu): \Gal(\eE^\ab / \eE) \to \eT(\bQ)^- \bs \eT(\bA_f),
\end{align*}
where $\eT(\bQ)^-$ is the closure of $\eT(\bQ)$ in $\eT(\bA_f)$. Fix now a neat compact open subgroup $\eK \subset \eT(\bA_f)$, and let $\textup{pr}_\eK$ be the projection $\eT(\bQ)^- \bs \eT(\bA_f) \to \eT(\bQ)\bs \eT(\bA_f) / \eK$. The kernel of $\textup{pr}_\eK$ is $\eT(\bQ)^- \bs \eT(\bQ)^-\eK$, so by definition $r(\mu)$ induces a map
\begin{align*}
	r(\mu)_\eK: \Gal(\eE^\ab / \eE_\eK) \to \eT(\bQ)^- \bs \eT(\bQ)^-\eK, 
\end{align*}
which we also denote by $r(\mu)_\eK$. 

Recall that $\eK^c$ denotes the image in $\eT^c(\bA_f)$ of any subgroup $\eK$ of $\eT(\bA_f)$. Note that if $\eK$ is neat, then $\eK^c$ is neat as well. By cuspidality of $\eT^c$, $\eT^c(\bQ)$ is discrete in $\eT(\bA_f)$ by \cite[Thm. 5.26]{Milne2005}, and $\eK^c \cap \eT^c(\bQ) = \{1\}$ (see also \cite[Lem. 1.5.5]{KSZ} and \cite[proof of Lem. 1.5.7]{KSZ}). Then $\eT \to \eT^c$ induces
\begin{align}
	\eT(\bQ)^- \bs \eT(\bQ)^-\eK \to \eT^c(\bQ) \bs \eT^c(\bQ)\eK^c \cong \eK^c
\end{align}
Denote by $r(\mu)_{\eK,p}$ the composition
\begin{align*}
	r(\mu)_{\eK,p}:\Gal(\eE^\ab / \eE_\eK) \xrightarrow{r(\mu)_\eK} \eT(\bQ)^- \bs \eT(\bQ)^-\eK \to \eK^c \hookrightarrow \eT^c(\bA_f) \to T^c(\bQ_p),
\end{align*}
where the last arrow denotes the projection. Let $v$ denote the place above $p$ induced by $\eE_\eK \hookrightarrow \overbar{\bQ} \hookrightarrow \overbar{\bQ}_p$, and let $E_\eK$ denote the completion of $\eE_\eK$ at $v$. Then the composition of $r(\mu)_{\eK,p}$ with the morphism of Galois groups
\begin{align*}
	\Gamma_{E_\eK} = \Gal(\overbar{\bQ}_p / E_\eK) \twoheadrightarrow \Gal(E^\ab / E_\eK) \to \Gal(\eE^\ab / \eE_\eK)
\end{align*}
induces a $T^c(\bQ_p)$-valued representation of $\Gamma_{E_\eK}$,
\begin{align}\label{eq-crystallinerep}
	r(\mu)_{\eK,p,\textup{loc}}: \Gamma_{E_\eK} \to T^c(\bQ_p).
\end{align}
Notice that if $\eK = \cT(\bZ_p)U^p$ for some small enough compact open subgroup $U^p$ of $\eT(\bA_f^p)$, then $r(\mu)_{\eK,p,\textup{loc}}$ actually lands in $\cT^c(\bZ_p)$, so in this case it induces a $\cT^c$-valued representation of $\Gamma_{E_\eK}$.

Suppose now $\rho: T^c \to \GL(W)$ is an algebraic representation of $T^c$ on a finite dimensional $\bQ_p$-vector space $W$. By composition with $\rho$, $r(\mu)_{\eK, p, \textup{loc}}$ induces a representation of $\Gamma_{E_\eK}$
\begin{align*}
	r(\mu,\rho)_{\eK,p}: \Gamma_{E_\eK} \to T^c(\bQ_p) \to \GL(W).
\end{align*}

\begin{lemma} \label{lem-crystalline}
	Assume that $\eK = \eK_p\eK^p$ where $\eK^p$ is a neat compact open subgroup of $\eT(\bA_f^p)$. For any representation $\rho: T^c \to \GL(W)$ of $T^c$ on a finite dimensional $\bQ_p$-vector space $W$, the induced $\Gamma_{E_\eK}$-representation $r(\mu,\rho)_{\eK,p}$ is crystalline. 
\end{lemma}
\begin{proof}
	This is a straightforward modification of \cite[Lem. 4.4]{KSZ}. Let us indicate the main points. 
		By a standard result in $p$-adic Hodge theory (see \cite[Prop. B.4 (i)]{Conrad2011} and the subsequent remark), to show $r(\mu,\rho)_{\eK, p}$ is crystalline it is enough to show that the composition 
		\begin{align}\label{crystallinelem-comp}
				\mathcal{O}_{E_\eK}^\times \hookrightarrow E_\eK^\times \to \Gamma_{E_\eK}^\ab \xrightarrow{r(\mu,\rho)_{\eK,p}} \GL(W)
			\end{align}
		agrees with the restriction to $\mathcal{O}_{E_\eK}^\times$ of an algebraic $\bQ_p$-group homomorphism $\Res_{E_\eK/ \bQ_p} \bG_m \to \GL(W)$.
		
		Let $f$ be the homomorphism of topological groups 
		\begin{align*}
				f: E_\eK^\times \xrightarrow{\textup{Art}_{E_\eK}} \Gamma_{E_\eK}^\ab \to \Gal(\eE^\ab / \eE_\eK) \to (\eK \cap \eT(\bQ)^-) \bs \eK,
			\end{align*}
		where $\textup{Art}_{E_\eK}$ is the local Artin map for $E_\eK$ (normalized geometrically), and denote by $f^c$ the composition
		\begin{align*}
				f^c: E_\eK^\times \xrightarrow{f} \eT(\bQ)^- \bs \eT(\bQ)^-\eK \to \eK^c \hookrightarrow \eT^c(\bA_f).
			\end{align*}
		By local-global compatibility, (\ref{crystallinelem-comp}) agrees with the composition
		\begin{align*}
				\mathcal{O}_{E_\eK}^\times \hookrightarrow E_\eK^\times \xrightarrow{f^c} T^c(\bA_f) \to T^c(\bQ_p) \xrightarrow{\rho} \GL(W),
			\end{align*}
		so it is enough to show that $f^c$ restricted to $\mathcal{O}_{E_\eK}^\times$ is algebraic.
		
		Let $f_1$ be the homomorphism of $\bQ_p$-groups
		\begin{align*}
				f_1:\Res_{E_\eK / \bQ_p} \bG_m \xrightarrow{\textup{Nm}_{E_\eK / E}} \Res_{E/\bQ_p} \bG_m \hookrightarrow (\Res_{\eE/\bQ} \bG_m) \otimes_{\bQ} \bQ_p \xrightarrow{r(\mu)^\textup{alg} \otimes_\bQ \bQ_p} T,
			\end{align*}
		and let $f_1^c$ be the composition of $f_1$ with $T \to T^c$. Applied to $\bQ_p$-points, $f_1^c$ induces a group homomorphism $E_\eK^\times \to T^c(\bQ_p)\hookrightarrow \eT^c(\bA_f)$. By construction (and local-global compatibility, once more), for any $x \in E_\eK^\times$, the images of $f_1(x)$ and $f(x)$ agree in $\eT(\bQ)^- \bs \eT(\bA_f)$, and therefore the images of $f_1^c(x)$ and $f^c(x)$ agree in $\eT^c(\bQ) \bs \eT^c(\bA_f)$.
		
		From here the arguments of \cite[Lem. 4.4]{KSZ} go through verbatim: If $x \in \mathcal{O}_{E_\eK}^\times$, cuspidality of $T^c$ forces the images of $f_1^c(x)$ and $f^c(x)$ to agree in $T^c(\bA_f)$, and the result follows.
\end{proof}

\subsection{The shtuka on the generic fiber}\label{sub-locsys}
In this section we return to the general setup at the beginning of section \ref{sub-SV} in order to explain how Pappas and Rapoport construct a shtuka over the generic fiber of a given Shimura variety with parahoric level structure.

Let $(\eG,X)$ be a Shimura datum with reflex field $\eE$. Let $v$ denote a place of $\eE$ corresponding to an embedding $\eE \hookrightarrow \bar{\bQ} \hookrightarrow \bar{\bQ}_p$, and let $E$ denote the completion of $\eE$ at $v$. If $L$ is an extension of $\bQ$, denote by $\mathcal{C}(L)$ the set of $\eG(L)$-conjugacy classes of cocharacters $\bG_{m,L} \to \eG_L$. By \cite[Lem. 1.1.3]{Kottwitz1984}, the embedding $\bar{\bQ} \to \bar{\bQ}_p$ induces a bijection $\mathcal{C}(\bar{\bQ}) \to \mathcal{C}(\bar{\bQ}_p)$, so the $G(\bar{\bQ})$-conjugacy class of characters $\{\mu_h\}$ coming from $(\eG, X)$ determines a unique $G(\bar{\bQ}_p)$-conjugacy class of cocharacters $\{\mu\}$ of $G_{\bar{\bQ}_p}$. One can check that the map $\mathcal{C}(\bar{\bQ}) \to \mathcal{C}(\bar{\bQ}_p)$ is equivariant for the action of $\Gal(\bar{\bQ}_p / \bQ_p)$, and it follows that $E$ is the local reflex field for $(G,\{\mu\})$, i.e., it is the field of definition for the conjugacy class $\{\mu\}$.

For any neat compact open subgroup $\eK \subset \eG(\bA_f)$, let 
\begin{align*}
	\Sh_\eK(\eG,X)_E = \Sh_\eK(\eG,X)_\eE \otimes_\eE E.
\end{align*}
Suppose $\eK = \eK_p \eK^p$, with $\eK_p = \cG(\bZ_p)$ a parahoric subgroup, and $\eK^p \subset \eG(\bA_f^p)$ neat. For each $\eK_p' \subset \eK_p$, denote by $\eK'$ the product $\eK_p'\eK^p$. Then the morphism
\begin{align}\label{covers}
	\Sh_{\eK'}(\eG,X)_E \to \Sh_{\eK}(\eG,X)_E
\end{align}
is a finite \'etale Galois cover, whose Galois group we denote by $\Gal(\Sh_{\eK'}/ \Sh_\eK)$. We also consider the infinite level Shimura variety
\begin{align*}
	\Sh_{\eK^p}(\eG,X)_E := \varprojlim_{\eK_p' \subset \eK_p} \Sh_{\eK_p'\eK^p}(\eG,X)_E.
\end{align*}
Since the transition morphisms (\ref{covers}) are affine, this limit can be taken in the category of schemes, by \cite[\href{https://stacks.math.columbia.edu/tag/01YX}{Tag 01YX}]{stacks-project}. Define
\begin{align*}
	\Gal(\Sh_{\eK^p} / \Sh_\eK) = \varprojlim_{\eK_p' \subset \eK_p} \Gal(\Sh_{\eK'}/\Sh_{\eK}).
\end{align*}
Then $\Sh_{\eK^p}(\eG,X)_E \to \Sh_{\eK}(\eG,X)_E$ is a torsor for the profinite group $\Gal(\Sh_{\eK^p}/\Sh_\eK)$ in the pro-\'etale topology on the scheme $\Sh_\eK(\eG,X)_E$. We denote this pro-\'etale $\Gal(\Sh_{\eK^p}/\Sh_\eK)$-torsor by $\bP_{0,\eK}$. We can make $\Gal(\Sh_{\eK^p}/\Sh_{\eK})$ explicit as follows. Denote by $\eZ(\bQ)^-$ the closure of $\eZ(\bQ)$ in $\eZ(\bA_f)$, and by $\eZ(\bQ)^-_\eK$ the intersection $\eZ(\bQ)^-\cap \eK$. Then 
\begin{align*}
	\Gal(\Sh_{\eK^p} /\Sh_{\eK}) = \eK_p / \eZ(\bQ)_{\eK,p}^-,
\end{align*}
 where $\eZ(\bQ)_{\eK,p}^-$ is the closure of the image of $\eZ(\bQ)^-_\eK$ under the projection $\eG(\bA_f) \to \eG(\bQ_p)$, see \cite[\textsection 1.5.8]{KSZ}.
 
 Since $\eK^p$ is neat, \cite[Lem. 1.5.7]{KSZ} implies that $\eZ(\bQ)_{\eK,p}^- \subset \eZ_{ac}(\bQ_p)$. Hence $\eK_p \to G^c(\bQ_p)$ factors through $\eK_p/\eZ(\bQ)_{\eK,p}^-$, and therefore $\Gal(\Sh_{\eK^p}/\Sh_{\eK})$ admits $\eK_p^c$ as a quotient. In particular, the composition (\ref{composition}) can be extended to
 \begin{align}\label{extendedcomposition}
 	\cG(\bZ_p) = \eK_p \twoheadrightarrow \Gal(\Sh_{\eK^p}/\Sh_\eK) \twoheadrightarrow \eK_p^c \hookrightarrow \cG^c(\bZ_p).
 \end{align}
However, note that the second surjection is never injective if $\eZ_{ac}$ is nontrivial, contrary to the claim in \cite[\textsection III, Rmk. 6.1]{Milne90}. For details, see \cite[(2)]{LSerrata}.
 
Let $\Sh(\eG,X)_E^\ad$ be the analytic adic space associated to $\Sh(\eG,X)_E$, and let $\bP_{0,\eK}^\ad$ be the pro-\'etale $\Gal(\Sh_{\eK^p}/\Sh_\eK)$-torsor on $\Sh(\eG,X)_E^\ad$ obtained from $\bP_{0,\eK}$. We would like to associate to $\bP_{0,\eK}^\ad$ a ``$\Gal(\Sh_{\eK^p}/\Sh_\eK)$-shtuka''. However, our machinery for shtukas only applies to groups which are algebraizable, and this is not a priori the case for $\Gal(\Sh_{\eK^p}/\Sh_\eK)$. Instead, we will replace $\Gal(\Sh_{\eK^p}/\Sh_{\eK})$ by the closely related group $\cG^c(\bZ_p)$, which \textit{is} algebraizable. Essentially, we view $\cG^c(\bZ_p)$ as a sort of ``algebraizable hull'' of $\Gal(\Sh_{\eK^p}/\Sh_\eK)$. In fact, at least when $\cG$ is reductive (and likely more generally), this can be made precise; see \cite[Prop. 4.8]{IKY}.

 Using the final two maps in (\ref{extendedcomposition}) we define the contracted product 
\begin{align}\label{eq-cover}
	\bP_{\eK} := \bP_{0,\eK}^\ad \times^{\Gal(\Sh_{\eK^p}/\Sh_{\eK})} \cG^c(\bZ_p),
\end{align} 
which is a pro-\'etale $\cG^c(\bZ_p)$-torsor $\bP_{\eK}$ on $\Sh_{\eK}(\eG,X)_E^\ad$. From every finite-dimensional $\bQ_p$-representation of $G^c$, we can obtain from $\bP_\eK$ a $\bZ_p$-local system on $\Sh(\eG,X)_E^\ad$, as follows. Let $\rho: G^c \to \GL(W)$ of $G^c$ be a finite dimensional $\bQ_p$-representation of $G^c$, and fix a $\bZ_p$-lattice $\Lambda \subset W$ such that $\rho(\cG^c(\bZ_p)) \subset \GL(\Lambda)$. The contracted product 
\begin{align*}
	\bL_{\rho,\Lambda}:= \bP_\eK \times^{\ul{\cG^c(\bZ_p)}} \ul{\Lambda}
\end{align*}
is a pro-\'etale $\bZ_p$-local system on $\Sh_\eK(\eG,X)_E^\ad$.  Denote by $\mu^c$ the cocharacter given by the composition of $\mu$ with $G \to G^c$.
 
% Following \cite{LS2018}, we can make $\bL_{\rho,\Lambda}$ more explicit, as follows. For each integer $r>0$, let $\eK_p^{(n)} = \ker(\cG(\bZ_p) \to \cG(\bZ/p^r\bZ))$, and let $\eK^{(n)} = \eK_p^{(n)}\eK^{p}$. Then
% \begin{align*}
% 	\Sh_{\eK^{(n)}}(\eG,X)_E \to \Sh_{\eK}(\eG,X)_E
% \end{align*}
%is a finite Galois cover, and the Galois group $\Gal(\Sh_{\eK^{(n)}}/\Sh_\eK)$ is a quotient of $\eK_p / \eK_p^{(n)}$. Define similarly $\eK_p^{c,\,(n)}:= \ker(\cG^c(\bZ_p) \to \cG^c(\bZ/p^n\bZ))$. For every $m > 0$, there is some $n(m) > 0$ with the property that $\eK_p^{c,\, (n(m))}$ (and therefore also $\eK_p^{(n(m))}$) acts trivially on $\Lambda / p^{(n(m))} \Lambda$. Let $\underline{\Lambda}_{p^{n(m)}}$ denote the constant $\bZ$-group scheme corresponding to $\Lambda / p^{n(m)} \Lambda$, and let $\bL_{\rho, \Lambda, m}$ be the \'etale sheaf of sections of the contracted product 
%\begin{align*}
%	\Sh_{\eK^{(n(m))}}(\eG,X)_E \times^{\eK_p/\eK_p^{(n(m))}} \underline{\Lambda}_{p^{n(m)}}.
%\end{align*}
%The limit $\varprojlim_m \bL_{\rho,\Lambda,m}$ 
%is a pro-\'etale $\bZ_p$-local system on $\Sh_\eK(\eG,X)_E$, and the corresponding $\bZ_p$-local system on $\Sh_\eK(\eG,X)_E^\ad$ is canonically identified with $\bL_{\rho,\Lambda}$.
 
 \begin{prop}\label{prop-genericfibershtuka}
 	There exists a $\cG^c$-shtuka $\sP_{\eK,E}$ over $\Sh_{\eK}(\eG,X)^\lozenge_E \to \Spd(E)$ with one leg bounded by $\mu^c$ which is associated to $\bP_{\eK}$ in the sense of \textup{\cite[\textsection 2.5]{PR2021}}. Furthermore, $\sP_{\eK,E}$ supports prime-to-$p$ Hecke correspondences.
 \end{prop}
\begin{proof}
	When $\eG = \eG^c$ (i.e., $\eZ_s$ is trivial), we have $\Gal(\Sh_{\eK^p}/\Sh_\eK) = \cG(\bZ_p) = \cG^c(\bZ_p)$, and in this case the proposition reduces to \cite[Prop. 4.1.2]{PR2021}. In general it follows similarly once one proves that $\bP_\eK$ is a de Rham $\cG^c(\bZ_p)$-torsor in the sense of \cite[Def. 2.5.5]{PR2021} and that the Hodge-Tate cocharacter for $\bP_\eK$ is given by $\mu^c$. 
	
	Fix a representation $(\Lambda, \rho)$ of $\eG^c$. To prove $\bL_{\rho,\Lambda}$ is de Rham, one proceeds in the same way as the proof of \cite[Thm. 1.2]{LZ17}. In particular, by \cite[Thm. 3.9(iv)]{LZ17} and the fact that the set of special points in each connected component of $\Sh_\eK(\eG,X)$ is non-empty \cite[Lem. 13.5]{Milne2005}, we reduce to showing that if $x=[h,a]_\eK\in\Sh_\eK(\eG,X)$ is a special point with $h$ factoring through a maximal torus $\eT \subset \eG$, then the stalk $\bL_{\rho,\Lambda, \bar{x}}$ of $\bL_{\rho,\Lambda}$ at a geometric point $\bar{x}$ above $x$ is de Rham. 
	
	Since $\eT_\bR$ stabilizes the special point $x$, it is $\bR$-anisotropic modulo the center of $\eG$, and therefore the restriction to $\eT$ of any representation of $\eG^c$ factors through $\eT^c$. If $\rho$ is a representation of $\eG^c$, let $\rho'$ denote its restriction to $\eT^c$. As in \cite[Lem. 4.8]{LZ17}, we see that $\bL_{\rho,\Lambda,\bar{x}}$ is isomorphic to $r(\mu,\rho')_{\eK,p}$. Hence the fact that $\bL_{\rho,\Lambda,\bar{x}}$ is de Rham follows from Lemma \ref{lem-crystalline}.
	
	Finally, for the boundedness by $\mu^c$, we first observe that the arguments of \cite[Prop. 4.3.14]{KSZ} show that the Hodge-Tate cocharacter of $r(\mu, \rho')_{\eK,p}$ is $\mu^c$. Thus the same is true of $\bL_{\rho,\Lambda, \bar{x}}$, independent of the choice of $x$. Therefore by \cite[Prop. 2.5.3 and Def. 2.6.6]{PR2021} the $\cG^c$-shtuka $\sP_{\eK,E}$ is bounded by $\mu^c$.

%	This doesn't follow exactly from \cite{LZ17}, since $\eG^c$ there is the quotient of $\eG$ by the maximal anisotropic torus of $\eZ^\circ$ which is split over $\bR$ (this coincides with $\eZ_{ac}$ when $\eZ^\circ$ splits over a CM extension of $\eE$, but not in general, see \cite[footnote 7]{KSZ}). However, the arguments which prove \cite[Thm. 1.2]{LZ17} work just as well in this case, with Lemma \ref{lem-crystalline} replacing \cite[Lem. 4.4]{LZ17}. Hence the result follows from this analog of \cite[Thm. 1.2]{LZ17} along with \cite[Thm. 3.9(iv)]{LZ17}.
%	
%	In general, for each $\rho$ and $\Lambda$ , the explicit description for $\bL_{\Lambda,\rho}$ given above shows that $\bL_{\Lambda,\rho}\otimes_{\bZ_p} \bQ_p$ agrees with the \'etale $\bQ_p$-local system $\bL_{\rho,W}$ described in \cite[\textsection 4.2]{LZ17}. Hence it follows from \cite[Thm. 1.2]{LZ17} combined with \cite[Thm. 3.9(iv)]{LZ17} that $\bP_\eK$ is a de Rham $\cG^c(\bZ_p)$-torsor in the sense of \cite[Def. 2.5.5]{PR2021}, and that the Hodge-Tate cocharacter for $\cG^c(\bZ_p)$ is given by $\mu^c$. Then $\sP_{\eK,E}$ is the $\cG^c$-shtuka associated to $\bP_\eK$ by \cite[Def. 2.5.7]{PR2021}. \edit{fix this proof now that we aren't assuming the center splits over a CM extension}
\end{proof}

\subsection{The conjecture of Pappas and Rapoport} \label{section-conj}
In this section we state the conjecture of Pappas and Rapoport on the existence of canonical integral models for Shimura varieties. Let $(G,b,\mu)$ be a local Shimura datum in the sense of \cite[Def. 24.1.1]{SW2020}, that is, suppose $G$ is a reductive $\bQ_p$-group, $\mu$ is a geometric conjugacy class of minuscule cocharacters, and $b$ is a $\sigma$-conjugacy class of elements of $G(\breve{\bQ}_p)$, which we further assume lies in the set $B(G,\mu^{-1})$ of neutral acceptable elements in the sense of \cite[Def. 2.3]{RV2014}. We take also a parahoric integral model $\cG$ for $G$ over $\bZ_p$. We write $E$ for the reflex field of $\mu$, and $k$ for an algebraic closure of the residue field of $E$

We write $\cM^\textup{int}_{\cG,b,\mu}$ for the \textit{integral local Shimura variety associated with $(\cG,b,\mu)$} \cite[\textsection 25.1]{SW2020}. This is the functor $\cM^\textup{int}_{\cG,b,\mu}$ on $\Perf_k$ that sends $S$ to the set of isomorphism classes of tuples $(S^\sharp, \sP, \phi_\sP, i_r)$, where $S^\sharp$ is an untilt of $S$ over $\Spa(\cO_{\breve{E}})$, $(\sP, \phi_{\sP})$ is a $\cG$-shtuka over $S$ with one leg along $S^\sharp$ bounded by $\mu$, and $i_r$ is an isomorphism of $G$-torsors
\begin{align*}
	i_r: G \res_{\cY_{[r,\infty)}(S)} \xrightarrow{\sim} \sP\res_{\cY_{[r,\infty)}(S)}
\end{align*}
for large enough $r$, under which $\phi_\sP$ is identified with $\phi_b:= b \times \Frob_S$. Here an isomorphism of tuples $(S^\sharp, \sP, \phi_{\sP}, i_{r_1}) \xrightarrow{\sim} ((S')^\sharp, \sP', \phi_{\sP'}, i'_{r_2})$ is an pair of isomorphism $(\alpha, \beta)$ where $\alpha: S^\sharp \xrightarrow{\sim} (S')^\sharp$ is an isomorphism whose tilt factors as $(S^\sharp)^\flat \xrightarrow{\sim} S \xleftarrow{\sim} ((S')^\sharp)^\flat$, and where $\beta: (\sP, \phi_{\sP}) \xrightarrow{\sim} (\sP', \phi_{\sP'})$ is an isomorphism of $\cG$-shtukas which satisfies $i'_r \circ i_r^{-1} = \beta$ for $r \ge \max(r_1, r_2)$. 

When $p \ne 2$ and $(G,b,\mu)$ is of abelian type, the main result of \cite{PR2022} states that $\cM^\textup{int}_{\cG,b,\mu}$ is representable by a formal scheme $\mathscr{M}_{\cG,b,\mu}$ which is normal and flat and formally locally of finite type over $\cO_{\breve{E}}$ (in fact, this holds slightly more generally, see \cite{PR2022}). 

The formation of integral local Shimura varieties is functorial. Indeed, let $(G,b,\mu) \to (G',b',\mu')$ be a morphism of local Shimura data, and suppose the map $G \to G'$ extends to $\cG \to \cG'$. Then we obtain a morphism of integral local Shimura varieties
\begin{align}\label{eq-functorialityILSV}
	\rho: \cM^\textup{int}_{\cG,b,\mu} \to \cM^\textup{int}_{\cG',b',\mu'} \times_{\Spd(\cO_{\breve{E}'})} \Spd(\cO_{\breve{E}})
\end{align}
given on points by pushing forward shtukas along $\cG \to \cG'$. Moreover, the morphism $\rho$ has the following property. Suppose $\sP^\univ$ and $(\sP')^{\univ}$ are the universal shtukas over $\cM^\textup{int}_{\cG,b,\mu}$ and $\cM^\textup{int}_{\cG,b,\mu}$, respectively. Then 
\begin{align}\label{eq-univpullback}
	\rho^\ast (\sP')^{\univ} = \sP^\univ \times^{\cG} \cG'.
\end{align}
Indeed, this follows from the definition of (\ref{eq-functorialityILSV}) and those of the respective universal objects.

Let $(G,b,\mu)$ be a local Shimura datum, and let $\cG$ be a parahoric model for $G$ over $\bZ_p$. We denote by $X_\cG(b,\mu^{-1})$ the \textit{$(b,\mu^{-1})$-admissible locus} in the Witt vector affine Grassmannian. This is the functor which assigns to perfect $k$-algebras $R$ the set of isomorphism classes of pairs $(\cP,\alpha)$, where $\cP$ is a $\cG$-torsor over $\Spec(W(R))$, and $\alpha$ is an isomorphism of $\cG$-torsors 
	\begin{align*}
		\alpha: \cG_{W(R)[1/p]} \xrightarrow{\sim} \cP_{W(R)[1/p]}
	\end{align*}
over $W(R)[1/p]$ such that $\phi_\cP = \alpha \circ \phi_b \circ \phi^\ast(\alpha)^{-1}$ defines the structure of a meromorphic Frobenius crystal (in the sense of \cite[Def. 2.3.3]{PR2021}), such that the $\cG$-shtuka $(\sP,\phi_\sP)$ associated to $(\cP, \phi_{\cP})$ by \cite[Thm. 2.3.5]{PR2021} has a leg along the divisor $p=0$ and is bounded by $\mu$. We note that $X_\cG(b,\mu^{-1})$ is sometimes referred to as the affine Deligne-Lusztig variety associated to $\cG$, $b$, and $\mu^{-1}$, but we reserve that moniker for its $k$-points.

The functor $X_\cG(b,\mu^{-1})$ is representable by a perfect scheme which is locally perfectly of finite type over $k$ \cite[\textsection 3.3]{PR2021}. Moreover, by \cite[Prop. 2.30]{Gleason2022}, there is a natural isomorphism
\begin{align}\label{eq-reducedlocus}
	(\cM^\textup{int}_{\cG,b,\mu})_\red \cong X_\cG(b,\mu^{-1}),
\end{align}
where $(\cM^\textup{int}_{\cG,b,\mu})_\red$ denotes the reduced locus of the $v$-sheaf $\cM^\textup{int}_{\cG,b,\mu}$ in the sense of \cite{Gleason2020}. Let $K = \cG(\bZ_p)$, $\breve{K} = \cG(\breve{\bZ}_p)$, and let $\textup{Adm}^K(\mu^{-1})$ denote the $\mu^{-1}$-admissible locus in the sense of \cite{Rapoport2005}. Then, the isomorphism (\ref{eq-reducedlocus}) implies, in particular, that we have the identity
\begin{align*}
	\cM^\textup{int}_{\cG,b,\mu}(\Spd(k)) = X_\cG(b,\mu^{-1})(k) = \{g\breve{K} = G(\breve{\bQ}_p)/ \breve{K} \mid g^{-1}b\phi(g) \in \textup{Adm}^K(\mu^{-1})\}.
\end{align*}
In other words, $\cM^\textup{int}_{\cG,b,\mu}(\Spd(k))$ is the affine Deligne-Lusztig variety associated to $\cG$, $b$, and $\mu^{-1}$. If $b \in \textup{Adm}^K(\mu^{-1})$, then $\cM^\textup{int}_{\cG,b,\mu}$ has a canonical $\Spd(k)$-valued \textit{base point},
\begin{align}\label{eq-basepoint}
	x_0 \in \cM^\textup{int}_{\cG,b,\mu}(\Spd(k)),
\end{align} 
corresponding to the trivial coset in $G(\breve{\bQ}_p)/\breve{K}$. From the perspective of the moduli functor, the base point associates to $S$ in $\Perf_k$ the tuple $(S, \cG\res_{\cY_{[0,\infty)}},\phi_b, \id)$. 

By \cite[Thm 2.]{Gleason2020}, $\cM^\textup{int}_{\cG,b,\mu}$ is a kimberlite (see \textit{loc. cit.} for details on this terminology). In particular, there is a continuous specialization map
\begin{align*}
	\textup{sp}: |\cM^\textup{int}_{\cG,b,\mu}| \to |(\cM^\textup{int}_{\cG,b,\mu})_\red| = |X_\cG(b,\mu^{-1})|.
\end{align*}
The \textit{formal completion of $\cM^\textup{int}_{\cG,b,\mu}$ along a point $x \in \cM^\textup{int}_{\cG,b,\mu}(\Spd(k))$} is the sub-$v$-sheaf $\widehat{\cM^\textup{int}_{\cG,b,\mu}}_{/x}$ whose points for $S$ in $\Perfk$ are given by
\begin{align*}
		\widehat{\cM^\textup{int}_{\cG,b,\mu}}_{/x}(S) = \{y: S \to \cM^\textup{int}_{\cG,b,\mu} \mid \textup{sp} \circ y(|S|) \subset \{x\}\}.
\end{align*}	

Suppose $\rho:(G,b,\mu) \to (G',b',\mu')$ is a morphism of local Shimura data which induces an isomorphism $\rho_\ad: G_\ad \xrightarrow{\sim} G'_\ad$, and suppose moreover that $\rho$ extends to a morphism $\cG \to \cG'$ of parahoric models. If $\cG$ and $\cG'$ correspond to the same point in the common building of $G_\ad$ and $G'_\ad$, then by \cite[Prop. 5.3.1]{PR2022}, the morphism (\ref{eq-functorialityILSV}) induces an isomorphism
\begin{align}\label{eq-functorialitycompletions}
	\hat{\rho}: \widehat{\cM^\textup{int}_{\cG,b,\mu}}_{/x} \xrightarrow{\sim} \widehat{\cM^\textup{int}_{\cG',b',\mu'}}_{/\rho(x)} \times_{\Spd (\cO_{\breve{E}'})} \Spd(\cO_{\breve{E}})
\end{align}
for any point $x \in \cM^\textup{int}_{\cG,b,\mu}(\Spd(k))$. In particular, if $G = T$ is a torus, and $\cG = \cT$ is the identity component of the N\'eron model for $T$, then the structure morphism to $\Spd(\cO_{\breve{E}})$ defines an isomorphism
\begin{align}\label{eq-trivtorus}
	\widehat{\cM^\textup{int}_{\cT,b,\mu}}_{/x} \xrightarrow{\sim} \Spd(\cO_{\breve{E}})
\end{align}
for any $x \in \cM^\textup{int}_{\cT,b,\mu}$, by applying the functoriality (\ref{eq-functorialitycompletions}) to the morphism from $(T,b,\mu)$ to the trivial local Shimura datum. 

Let us now return to the general global setting of section \ref{sub-SV} in order to state the conjecture of Pappas and Rapoport \cite{PR2021}. We retain the notation from the beginning of section \ref{sub-SV} as well, so $(\eG, X)$ is a global Shimura datum, $\mu$ is the corresponding geometric conjugacy class of cocharacters for $\eG$, $\eE$ is the reflex field, $G = \eG_{\bQ_p}$ is the corresponding $p$-adic group, $E$ is the local reflex field, and $\eK = \eK_p \eK^p \subset \eG(\bA_f)$ is the level subgroup with $\eK_p = \cG(\bZ_p)$ is a parahoric subgroup and $\eK^p \subset \eG(\bA_f^p)$ is neat. We furthermore have the cuspidal quotient $\eG^c$, its corresponding $p$-adic group $G^c$, and the parahoric group scheme $\cG^c$. 

Suppose we are given a system of normal integral models $\sS_\eK$ of $\Sh_{\eK}(\eG,X)_E$, each equipped with a $\cG^c$-shtuka $\sP_\eK$ defined over $\sS_\eK^\lozenge \to \Spd(\cO_E)$ which is bounded by $\mu$, and which extends the $\cG^c$-shtuka $\sP_{\eK,E}$ over $\Sh_{\eK}(\eG,X)_E^\lozenge \to \Spd(E)$ defined in Proposition \ref{prop-genericfibershtuka}. Then for any point $x \in \sS_\eK(k)$, we have a canonically defined $\sigma$-conjugacy class $[b_x] \in B(G^c)$ coming from $\sP_\eK$. Indeed, by \cite[Ex. 2.4.9]{PR2021}, the pullback $x^\ast \sP_\eK$ defines a $\cG^c$-torsor $\sP_x$ over $\Spec(W(k))$ along with an isomorphism
\begin{align*}
	\phi_x: \phi^\ast (\sP_x)[1/p] \xrightarrow{\sim} \sP_x[1/p],
\end{align*}
where here $\phi$ denotes the Frobenius for $W(k)$. Then any choice of trivialization yields an element $b_x \in G^c(\breve{\bQ}_p)$, and changing the trivializating corresponds to applying $\sigma$-conjugation by some element of $\cG^c(\breve{\bZ}_p)$. Since the shtuka $\sP_\eK$ is bounded by $\mu^c$, the resulting element of $B(G^c)$ lies in $B(G^c,(\mu^c)^{-1})$. Denote by $x_0$ the base point of $\cM^\textup{int}_{\cG^c,b_x,\mu^c}(\Spd(k))$ as in \ref{eq-basepoint}. 

The following is a slightly modified version of \cite[Conj. 4.2.2]{PR2021}.

\begin{conj}[Pappas-Rapoport] \label{conj-PR}
	Fix a parahoric subgroup $\eK_p$ with corresponding  $\bZ_p$-group scheme $\cG$. Then for every neat compact open subgroup $\eK^p \subset \eG(\bA_f^p)$, there exists a normal flat model $\sS_\eK$ of $\Sh_\eK(\eG,X)_E$ over $\cO_E$, such that the system $(\sS_\eK)_{\eK^p}$ satisfies the following properties:
	\begin{enumerate}[(a)]
		\item For every discrete valuation ring $R$ of characteristic $(0,p)$ over $\cO_E$,
		\begin{align*}
			\left(\varprojlim_{\eK^p} \Sh_\eK(\eG,X)_E\right)(R[1/p]) = \left(\varprojlim_{\eK^p} \sS_\eK\right)(R).
		\end{align*}
		If $\Sh_\eK(G,X)_E$ is proper over $\Spec(E)$ then $\sS_\eK$ is proper over $\Spec(\cO_E)$. In addition, the system $\sS_\eK$ supports prime-to-$p$ Hecke correspondences, i.e., for $g \in \eG(\bA_f^p)$ and ${\eK'}^p$ with $g{\eK'}^pg^{-1} \subset \eK^p$, there are finite \'etale morphisms $[g]: \sS_{\eK'} \to \sS_\eK$ which extend the natural maps $[g]: \Sh_{\eK_p{\eK'}^p}(\eG,X)_E \to \Sh_{\eK_p\eK^p}(\eG,X)_E$. 
		\item The $\cG^c$-shtuka $\sP_{\eK,E}$ extends to a $\cG^c$-shtuka $\sP_\eK$ on $(\sS_\eK)^{\lozenge \slash}$.
		\item For $x \in \sS_\eK(k)$ and $b_x$ defined as above, there is an isomorphism of $v$-sheaves over $\Spd(\cO_E)$
		\begin{align*}
			\Theta_x: \widehat{\cM^\textup{int}_{\cG^c,b_x,\mu^c}}_{/x_0} \xrightarrow{\sim} (\widehat{{\sS_{\eK}}}_{/x})^\lozenge,
		\end{align*}
		under which the pullback shtuka $\Theta_x^\ast(\sP_\eK)$ coincides with the universal shtuka on $\cM^\textup{int}_{\cG^c,b_x,\mu^c}$. 
	\end{enumerate}
\end{conj}

In the case where $\eG$ is itself cuspidal (i.e., $\eG = \eG^c$), this is exactly \cite[Conj. 4.2.2]{PR2021}, which is proved in \textit{loc. cit.} for Shimura varieties of Hodge type, that is, those for which there is an embedding of $(\eG, X)$ into a Shimura datum for a Siegel-type Shimura datum. We note that cuspidality is automatic in the Hodge-type case by \cite[Lem. 5.1.2]{KSZ}, so the more general formulation of the conjecture is unnecessary for the results of \cite{PR2021}. For Shimura varieties of toral type (the case of interest below) or of abelian type, there is no reason to expect cuspidality in general. We have formulated the more general version of the conjecture with this in mind.

Pappas and Rapoport moreover prove that an integral model satisfying the properties of Conjecture \ref{conj-PR} is uniquely determined \cite[Thm. 4.2.4]{PR2021}.

\begin{thm} [Pappas-Rapoport]
	There is at most one system of normal flat models $\sS_\eK$ of $\Sh_\eK(\eG,X)_E$ over $\cO_E$, for $\eK = \eK_p \eK^p$ with variable neat $\eK^p$, which satisfies the properties in Conjecture \ref{conj-PR}. 
\end{thm}
\begin{proof}
	The proof in \cite[Thm. 4.2.4]{PR2021} addresses the case where $\eG = \eG^c$. The proof in general follows by replacing the integral local Shimura variety for $\cG$ with the one for $\cG^c$.
\end{proof}

\subsection{Proof of Theorem \ref{thm-A}} \label{section-tori}\label{sub-proof}
%
%Let $\eT$ be a $\bQ$-torus which splits over a CM field, and let $\eT^c = \eT/\eT_s$ be the quotient of $\eT$ by the maximal $\bQ$-anisotropic subtorus of $\eT$ which splits over $\bR$. A $\bQ$-torus is \textit{cuspidal} in the sense of \cite[Def. 1.5.4]{KSZ} if its $\bQ$-split rank is equal to its $\bR$-split rank. By the results of \cite[\textsection 1.5]{KSZ} (see in particular footnote 7), the torus $\eT^c$ is cuspidal.\edit{could write this without assuming that $\eT$ splits over a CM field, and then later specialize if needed}
%
%Let $h: \mathbb{S} \to \eT_\bR$ be a homomorphism of $\bR$-group schemes. Then the pair $(\eT, h)$ comprises a Shimura datum, and $h$ determines a conjugacy class of geometric cocharacters $\{\mu_h\}$ defined over the reflex field $\eE$ of $(\eT,h)$. 

In this section we prove Conjecture \ref{conj-PR} for Shimura varieties defined by tori. As in section \ref{sub-SV}, let $(\eT,\{h\})$ be a Shimura datum defined by a torus $\eT$ over $\bQ$, let $\eT^c$ be its maximal cuspidal quotient, and let $\eE$ denote the reflex field. Let $\eK = \eK_p \eK^p$ with $\eK^p$ neat, and $\eK_p = \cT(\bZ_p)$ for $\cT$ the identity component of the locally finite-type N\'eron model of $T = \eT_{\bQ_p}$, so $\eK_p$ is the unique parahoric subgroup of $T(\bQ_p)$. We obtain a zero-dimensional Shimura variety
\begin{align}\label{eq-Ei}
	\Sh_\eK(\eT,\{h\}) = \coprod_{i\in I} \Spec{\eE_i}, 
\end{align}
where each $\eE_i$ is a finite extension of $\eE$. Moreover, by Lemma \ref{lem-E_K}, we know that each $\eE_i$ is isomorphic to the field $\eE_\eK$ defined in (\ref{eq-EU}).

As in Section \ref{sub-locsys}, let $v$ be the place of $\eE$ above $p$ corresponding to our chosen embedding $\eE \hookrightarrow \bar{\bQ} \hookrightarrow \bar{\bQ}_p$, and let $E$ denote the completion of $\eE$ at $v$. Let $\mathcal{O}_E$ and $\cO_{E_i}$ denote the rings of integers in $E$ and $E_i$, respectively. Then 
\begin{align*}
	\sS_\eK(\eT,\{h\}) := \coprod_{i\in I} \Spec{\mathcal{O}_{E_i}}
\end{align*}
provides us with an integral model for $\Sh_\eK(\eT,\{h\})_E$.

We begin by using Lemma \ref{lem-crystalline} to extend the $\cT^c$-shtuka $\sP_{\eK,E}$ over $\Sh_\eK(\eT,\{h\})^\lozenge$ of Proposition \ref{prop-genericfibershtuka} to all of $\sS_\eK(\eT, \{h\})^{\lozenge \slash}$. By Lemma \ref{lem-crystalline}, the $\cT^c$-valued representation induced by the homomorphism
\begin{align}\label{eq-Trep}
	r(\mu)_{\eK,p,\textup{loc}}: \Gamma_{E_\eK} \to \cT^c(\bZ_p)
\end{align}
is crystalline. Define $\sP_{\eK,i}$ to be the $\cT^c$-shtuka over $\Spd (\cO_{E_i})$ associated by Definition \ref{def-greptoshtuka}  to the $\cT^c$-valued crystalline representation of $\Gamma_{E_i}$ given by (\ref{eq-Trep}). Since $\Spec(\cO_{E_i})$ is proper over $\Spec(\bZ_p)$, we have 
\begin{align*}
	\Spd(\cO_{E_i})	\xrightarrow{\sim} (\Spec(\cO_{E_i}))^{\lozenge \slash} 
\end{align*}
for every $i$, by \cite[Rmk. 4.6 (iv)]{Huber1994}, see (\ref{eq-slash}) and the discussion preceding it. By patching together the collection $\{\sP_{\eK,i}\}_{i\in I}$ we can define a $\cT^c$-shtuka $\sP_\eK$ over all of $\sS_\eK(\eT, \{h\})^{\lozenge \slash}$. 

\begin{prop}\label{prop-torusshtuka}
	The $\cT^c$-shtuka $\sP_\eK$ over $\sS_\eK(\eT,\{h\})^{\lozenge \slash}$ is bounded by $\mu^c$, and it extends the $\cT^c$-shtuka $\sP_{\eK,E}$ over $\Sh_\eK(\eT,\{h\})$. 
\end{prop}
\begin{proof}
	It's enough to show the proposition for each $i \in I$. Note that the Hodge-Tate cocharacter of each $r(\mu, \rho)_{\eK,p}$ is given by $\mu^c$ by \cite[Lem. 4.4.]{KSZ}, so the first point follows from Lemma \ref{lem-bdbymu}. 
	
	The second point follows from Lemma \ref{lem-genericfiber}, since for each $(\Lambda, \rho)$, the stalk of the $\bZ_p$-local system $\bL_{\rho,\Lambda}$ from Section \ref{sub-locsys} at a geometric point $\bar{x}_i$ over $\Spec(E_i)$ is given by the 
	representation $r(\mu,\rho)_{\eK,p}$ of $\Gamma_{E_\eK}$.
\end{proof}

This proves part (b) of Conjecture \ref{conj-PR} for $\Sh_\eK(\eT,\{h\})$. Next we prove part (c) of the conjecture. Let us recall the statement. Let $\sS_\eK = \sS_\eK(\eT,\{h\})$, let $x \in \sS_\eK(k)$, and let $b_x$ be defined as in Conjecture \ref{conj-PR}. We want to show there is an isomorphism
\begin{align*}
	\Theta_x: \widehat{\cM^\textup{int}_{\cT^c, b_x, \mu^c}}_{/x_0} \xrightarrow{\sim} (\widehat{{\sS_{\eK}}}_{/x})^\lozenge,
\end{align*}
of $v$-sheaves over $\Spd(\cO_{\breve{E}})$ such that the pullback $\Theta_x^\ast(\sP_\eK)$ coincides with the universal shtuka on $\cM^\textup{int}_{\cT^c, b_x,\mu^c}$. 

\begin{lemma}\label{lem-fmlcomp}
	Let $x \in \sS_\eK(k)$. Then there is a natural isomorphism
	\begin{align*}
		\Spd(\cO_{\breve{E}}) \xrightarrow{\sim} (\widehat{\sS_\eK}_{/x})^\lozenge.
	\end{align*}
\end{lemma}
\begin{proof}
	Recall that each $E_i$ is isomorphic to $E_\eK$ (see Lemma \ref{lem-E_K}), so it is enough to show $E_\eK$ is an unramified extension of $E$. By definition of $\eE_\eK$, we have
	\begin{align*}
		\Gal(E^\ab / E_\eK) = \ker(\Gal(E^\ab / E) \hookrightarrow \Gal(\eE^\ab / \eE) \xrightarrow{r(\mu)_\eK} \eT(\bQ) \bs \eT(\bA_f) / \eK).
	\end{align*}
	By local class field theory, if $E^\textup{unr}$ denotes the maximal unramified extension of $E$ in $\bar{\bQ}_p$, then the local Artin map restricts to an isomorphism $\cO_E^\times \xrightarrow{\sim} \Gal(E^\ab / E^\textup{unr})$. Therefore, to show $E_\eK$ is unramified, it is enough to show that $\cO_E^\times$ is in the kernel of the composition
	\begin{align}\label{eq-wantkernel}
		E^\times \hookrightarrow \eE^\times \bs \bA_\eE^\times \xrightarrow{\textup{Art}_\eE} \Gal(\eE^\ab / \eE) \xrightarrow{r(\mu)_\eK} \eT(\bQ) \bs \eT(\bA_f) / \eK.
	\end{align}
	The composition (\ref{eq-wantkernel}) can be rewritten as 
	\begin{align}\label{eq-rewrite}
		E^\times \hookrightarrow \eE^\times \bs \bA_\eE^\times \xrightarrow{r(\mu)^\textup{alg}(\bA)} \eT(\bQ) \bs \eT(\bA) \xrightarrow{\textup{pr}_\eK} \eT(\bQ) \bs \eT(\bA_f) / \eK
	\end{align}
	by the definition of $r(\mu)_\eK$. The composition of the first two arrows of (\ref{eq-rewrite}) is given by the evaluation on $\bQ_p$ of the homomorphism of $\bQ_p$-group schemes
	\begin{align}\label{eq-grpschm}
		\Res_{E/\bQ_p} \bG_m \hookrightarrow (\Res_{\eE/\bQ} \bG_m)_{\bQ_p} \xrightarrow{r(\mu)^\textup{alg}} T,
	\end{align}
	composed with $T(\bQ_p) \hookrightarrow \eT(\bA) \to \eT(\bQ) \bs \eT(\bA)$. The homomorphism (\ref{eq-grpschm}) necessarily sends $\cO_E^\times$ into $\cT(\bZ_p) \subset \eK$, since the evaluation of (\ref{eq-grpschm}) on $\bQ_p$-points is continuous, and $\cT(\bZ_p)$ is the unique maximal compact subgroup of $T(\bQ_p)$. The result follows.

%	 Then the ramification of $E_\eK$ is controlled by 
%	\begin{align*}
%		\ker(\cO_E^\times \to \eT^c(\bQ_p) / \cT^c(\bZ_p)),
%	\end{align*}
%	but this kernel is all of $\cO_E^\times$ because the reflex norm is a morphism of group schemes between tori, thus preserves their respective Iwahori subgroups. \edit{write out this whole proof}
\end{proof}

Combining Lemma \ref{lem-fmlcomp} with the isomorphism (\ref{eq-trivtorus}), we obtain an isomorphism
\begin{align}\label{eq-theta}
	\Theta_x: \widehat{\cM^\textup{int}_{\cT^c,b_x,\mu^c}}_{/x_0} \xrightarrow{\sim}  (\widehat{{\sS_{\eK}}}_{/x})^\lozenge.
\end{align}
We remark that the isomorphism in Lemma \ref{lem-fmlcomp} is, in fact, the inverse of the structure morphism for $(\widehat{{\sS_{\eK}}}_{/x})^\lozenge$ over $\Spd(\cO_{\breve{E}})$, so in particular $\Theta_x$ is a uniquely determine $\Spd(\cO_{\breve{E}})$-morphism. It remains to study the pullback of $\sP_{\eK}$ along $\Theta_x$. 

Recall that $\eT^c$ is the maximal cuspidal quotient of $\eT$. As in Section \ref{sub-SV}, the map $\eT \to \eT^c$ extends to $\cT \to \cT^c$, where $\cT^c$ is the connected N\'eron model of $\eT^c_{\bQ_p}$. Let $\eK' = \cT^c(\bZ_p)\eK^{p,c}$, where $\eK^{p,c}$ is the image of $\eK^p$ under $\eT \to \eT^c$, and let $\eE^c$ denote the reflex field of the Shimura datum $(\eT^c, \{h^c\})$, and $E^c$ denote the completion of $\eE^c$ at the place $v'$ induced by $v$. Let $\eE_i$ be defined as in (\ref{eq-Ei}), and let $E_i = \eE \otimes_{\eE} E$. Then, by the valuative criterion for properness of $\sS_{\eK'}$ over $\cO_{E^c}$, the morphism 
\begin{align*}
	\Spec(E_i) \hookrightarrow \Sh_\eK(\eT,\{h\})_E \to \Sh_{\eK'}(\eT^c,\{h^c\})\otimes_{\eE^c} E \to \sS_{\eK'}\otimes_{\cO_{\eE^c}} \cO_{E}
\end{align*}
extends to a morphism $\Spec(\cO_{E_i}) \to \sS_{\eK'}\otimes_{\cO_{\eE^c}} \cO_{E}$. Taken together for each $i$, we obtain a morphism of integral models 
\begin{align}\label{eq-intmorphism}
	\sS_\eK \to \sS_{\eK'} \otimes_{\cO_{E^c}} \cO_E.
\end{align}
As in the case of $\sS_\eK$, we have 
\begin{align*}
	\sS_{\eK'} = \coprod_i \Spec(\cO_{E_{\eK'}}),
\end{align*}
where $E_{\eK'}$ is defined as in (\ref{eq-EU}). We see that $E_{\eK'} \subset E_{\eK}$, hence $E_{\eK'}$ is a finite unramified extension of $E$, and $\cO_{\eK'} \to \cO_{\eK}$ is finite \'etale. It follows that, for any point $x \in \sS_\eK(k)$, we have a commutative diagram
\begin{equation}\label{diagram}
	\begin{tikzcd}
		& \Spd(\cO_{\breve{E}}) 
			\arrow[dl, "\sim"'] \arrow[dr, "\sim"]
		& \\
		(\widehat{\sS_\eK}_{/x})^\lozenge
			\arrow[rr, "\sim"]
		& & (\widehat{\sS_{\eK'}}_{/x'})^\lozenge, 
	\end{tikzcd}
\end{equation}
where $x'$ denotes the image of $x$ in $\sS_{\eK'}(k)$, and the bottom arrow is induced by (\ref{eq-intmorphism}).

We claim that the pullback of $\sP_{\eK'}$ along (\ref{eq-intmorphism}) is $\sP_\eK$. By \cite[Cor. 2.7.10]{PR2021}, it is enough to show that $\sP_{\eK', E'}$ pulls back to $\sP_{\eK,E}$, and by functoriality of the construction in \cite[Def. 2.6.6]{PR2021}, we are in turn reduced to showing that the pullback of the $\bP_{\eK'}$ is $\bP_\eK$, where $\bP_{\eK'}$ and $\bP_\eK$ are defined as in (\ref{eq-cover}). But this is straightforward to check from the definition of $\bP_\eK$, see e.g., \cite[(4.3.5)]{IKY}. 

Now, by the diagram  (\ref{diagram}), to show that $\sP_\eK$ pulls back to the universal shtuka along $\Theta_x$, it is enough to show the corresponding result for $\sP_{\eK'}$. Hence for the remainder of this section we assume $\eT = \eT^c$. 

The idea for computing the pullback of $\sP_\eK$ is to use functoriality to reduce the result to Lubin-Tate theory. Let $T_1$ be the $\bQ_p$-torus 
\begin{align*}
	T_1 = \Res_{E_\eK/\bQ_p} \bG_m
\end{align*}
with $\bZ_p$-model given by the Iwahori group scheme $\cT_1 = \Res_{\cO_{E_\eK}/\bZ_p} \bG_m$. We identify $(T_1)_{\overbar{\bQ}_p}$ with $\prod_{\tau \in \Hom_{\bQ_p}(E_\eK, \overbar{\bQ}_p)} \bG_{m,\overbar{\bQ}_p}$, and we define the cocharacter $\mu_1: \bG_m \to (T_1)_{\overbar{\bQ}_p}$ by $z \mapsto (z, 1,\dots, 1)$, with the first factor indexed by the chosen embedding $E_\eK \hookrightarrow \overbar{\bQ}_p$. Denote by $f_1$ the homomorphism of $\bQ_p$-group schemes given by
\begin{align*}
	f_1: T_1 = \Res_{E_\eK/\bQ_p} \bG_m \xrightarrow{N_{E_\eK/E}} \Res_{E/\bQ_p} \bG_m \hookrightarrow (\Res_{\eE/\bQ} \bG_m)_{\bQ_p} \xrightarrow{r(\mu)^\textup{alg}_{\bQ_p}} T,
\end{align*}
where $r(\mu)^\textup{alg}$ is the morphism defined in (\ref{eq-reflexnorm}). One can check that $\mu = f_1 \circ \mu_1$, so $f_1$ extends to a morphism of local Shimura data $(T_1, b_1, \mu_1) \to (T,b_x,\mu)$, where $b_1$ is the unique element of $B(T_1, \mu_1^{-1})$. Moreover, $f_1$ extends to a morphism $\cT_1 \to \cT$ by the N\'eron mapping property and functoriality of identity components, so by \cite[Prop. 5.3.1]{PR2022}, the morphism of $\Spd(\cO_{\breve{E}})$-$v$-sheaves $\cM^\textup{int}_{\cT_1,b_1,\mu_1} \to \cM^\textup{int}_{\cT,b_x,\mu}$ induces an isomorphism
\begin{align*}
	\hat{f_1}: \widehat{\cM^\textup{int}_{\cT_1,b_1,\mu_1}}_{/x_1} \xrightarrow{\sim} \widehat{\cM^\textup{int}_{\cT,b_x,\mu}}_{/x_0},
\end{align*}
where $x_1$ is the base point of $\cM^\textup{int}_{\cT_1,b_1,\mu_1}$ (see (\ref{eq-functorialitycompletions})). Hence we obtain a commutative diagram
\begin{equation} \label{bigdiagram}
	\begin{tikzcd}[row sep=tiny]
		& \widehat{\cM^\textup{int}_{\cT_1,b_1,\mu_1}}_{/x_1} 
			\arrow[r, hook] \arrow[dd,"\sim" labl]
		& \cM^\textup{int}_{\cT_1,b_1,\mu_1}
			\arrow[dd]
		\\ \Spd(\cO_{\breve{E}}) 
			\arrow[ur, "\Psi_{x,1}"] \arrow[dr, "\Psi_x"']
		& & \\
		& \widehat{\cM^\textup{int}_{\cT,b_x,\mu}}_{/x_0}
			\arrow[r, hook]
		& \cM^\textup{int}_{\cT,b_x,\mu},
	\end{tikzcd}
\end{equation}
where $\Psi_{x,1}$ and $\Psi_x$ are the inverses of the isomorphisms coming from (\ref{eq-trivtorus}). We remark that the left triangle commutes by because $\hat{f}_1$ is a morphism over $\Spd(\cO_{\breve{E}})$, and because $\Psi_{x,1}$ an $\Psi_x$ are the inverses of the structure morphisms.

Denote by $\sP^\univ$ and $\sP_1^\univ$ the universal shtukas on $\cM^\textup{int}_{\cT,b_x,\mu}$ and $\cM^\textup{int}_{\cT_1,b_1,\mu_1}$, respectively, and let $\sP_{\eK,x}$ denote the restriction of $\sP_\eK$ to $(\widehat{\sS_\eK}_{/x})^\lozenge$. Since $\Theta_x$ is an isomorphism, to show $\Theta_x^\ast(\sP_{\eK,x}) \cong \sP^\univ$ it is enough to show $\Psi_x^\ast(\sP^\univ) \cong \sP_{\eK,x}$. By (\ref{eq-univpullback}), we have $\hat{f_1}^\ast (\sP^\univ) \cong \sP_1^\univ \times^{\cT_1} \cT$, so by (\ref{bigdiagram}) we see that
\begin{align}\label{eq-univshtuka}
	\Psi_x^\ast (\sP^\univ) \cong \Psi_{x,1}^\ast(\sP_1^\univ) \times^{\cT_1} \cT.
\end{align}

Let $\Gamma_{E_\eK,0}$ denote the inertia subgroup of $\Gamma_{E_\eK}$, and let
\begin{align}\label{eq-LT}
	\textup{LT}: \Gamma_{E_\eK,0} \to \cO_{E_\eK}^\times
\end{align}
denote the Lubin-Tate character \cite[\textsection 2.1]{Serre1979} Then $\textup{LT}$ induces a $\cT_1$-valued crystalline representation $\alpha_0$ of $\Gamma_{E_\eK,0}$. By the arguments in \cite[Prop. 4.3.14]{KSZ}, the restriction of $r(\mu)_{\eK,p,\textup{loc}}$ (see (\ref{eq-crystallinerep})) to $\Gamma_{E_\eK,0}$ is given by the composition 
\begin{align*}
	\Gamma_{E_\eK,0} \xrightarrow{\textup{LT}} \cO_{E_\eK}^\times = \cT_1(\bZ_p) \xrightarrow{f_1} \cT(\bZ_p).
\end{align*}
Denote by $\sP_1$ the $\cT_1$-shtuka associated to $\alpha_0$ by Definition \ref{def-greptoshtuka}. Then Lemma \ref{lem-functorialities} implies that there is a natural isomorphism
\begin{align}\label{eq-shtukas}
	\sP_1 \times^{\cT_1} \cT \xrightarrow{\sim} \sP_{\eK,x}.
\end{align}
By combining (\ref{eq-univshtuka}) and (\ref{eq-shtukas}), we see that we are reduced to showing the following proposition.

\begin{prop}\label{prop-LT}
	There is a natural isomorphism of shtukas
	\begin{align*}
		\Psi_{x,1}^\ast(\sP_1^\univ) \cong \sP_1.
	\end{align*}
\end{prop}
\begin{proof}
	First note that the local Shimura datum $(T_1, b_1,\mu_1)$ along with the parahoric group scheme $\cT_1$ actually come from an RZ-datum of EL-type (see \cite[\textsection 24.3]{SW2020}). Fix a uniformizer $\pi$ of $\cO_{E_\eK}$. In this case, the RZ-datum is given by the tuple $\mathcal{D} = (E_\eK, E_\eK, \cO_{E_\eK}, \mathcal{L})$, where $\mathcal{L}$ is the lattice chain given by multiples of $\cO_{E_\eK}$, i.e., $\mathcal{L} =\{ \pi^k\cO_{E_\eK}\}_{k \in \mathbb{Z}}$. Denote by $\cM_{\LT}$ the Rapoport-Zink formal scheme associated to $\mathcal{D}$. By \cite[Cor. 25.1.3]{SW2020}, there is a natural isomorphism
	\begin{align*}
		(\cM_{\LT})^\lozenge \xrightarrow{\sim} \cM^\textup{int}_{\cT_1, b_1, \mu_1}.
	\end{align*}
	By \cite[Prop. 18.4.1]{SW2020}, the morphism 
	\begin{align*}
		\Spd(\cO_{\breve{E}}) \xrightarrow{\Psi_{x,1}} \widehat{\cM^\textup{int}_{\cT_1,b_1,\mu_1}}_{/x_1} \to \cM^\textup{int}_{\cT_1,b_1,\mu_1}
	\end{align*}
	is induced by a unique morphism of formal schemes over $\Spf(\cO_{\breve{E}})$
	\begin{align}\label{eq-fmlschhom}
		\Spf(\cO_{\breve{E}}) \to \cM_\LT.
	\end{align}
	By \cite[3.78]{RZ1996}, for any locally $p$-nilpotent $\cO_{\breve{E}}$-scheme $S$, $\cM_\LT(S)$ parametrizes pairs $(X,\rho)$ consisting of a $p$-divisible group over $S$ with an action of $\cO_E$, such that the induced action of $\cO_E$ on $\textup{Lie }X$ agrees with the natural one, and $\rho$ is a quasi-isogeny between the $p$-divisible group $X_{b_1}$ associated to $b_1$ by Dieudonn\'e theory and $X$ modulo $p$. In particular, the morphism (\ref{eq-fmlschhom}) determines a $p$-divisible group $X$ over $\Spf(\cO_{\breve{E}})$ (equivalently, over $\Spec(\cO_{\breve{E}})$, see \cite[Lem. 2.4.4]{deJong1996}), and the shtuka $\Psi_{x,1}^\ast(\sP_1^\univ)$ is given by the shtuka associated to $X$ by \cite[Ex. 2.3.2]{PR2021}.
	
	By deformation theory \cite[Prop. 4.2]{DrinfeldEllipticModules}, up to isomorphism there is a unique $p$-divisible group $X$ over $\Spec(\cO_{\breve{E}})$ as in the preceding paragraph. Hence $X$ is given by the base change to $\Spf(\cO_{\breve{E}})$ of the Lubin-Tate formal group $X_\LT$ over $\Spec(\cO_{E_\eK})$. It follows that the representation of $\Gamma_{E_\eK,0}$ given by the Tate module of $X$ is the restriction of the representation of $\Gamma_{E_\eK}$ given by the Tate module of $X_\LT$. On the other hand, by \cite[A.4 Prop. 4]{Serre68}, the latter is given by the Lubin-Tate character, $\LT$ (see (\ref{eq-LT})). Therefore, to complete the proof of the proposition, it remains only to show that, for any $p$-divisible group $X$ over $\Spec(\cO_{\breve{E}})$, the shtuka associated to $X$ by \cite[Ex. 2.3.2]{PR2021} is the same as that associated to the crystalline representation given by the Tate module of $X$ by Definition \ref{cons-crystaltoshtuka}.
	
	Let $\cM_\Prism(X)$ be the prismatic Dieudonn\'e crystal associated to $X$ by \cite{ALB}. By \cite[Prop. 4.3.7]{ALB}, the shtuka associated to $X$ by \cite[Ex. 2.3.2]{PR2021} is given by the pullback of the BKF-module over $\Spd(\cO_{\breve{E}})$ given by restricting $\cM_\Prism(X)$ to perfect prisms as in Section \ref{sub-prismtoshtuka}. On the other hand, by \cite[Prop. 3.34]{DLMS}, the prismatic $F$-crystal associated to $T_p(X)$ is exactly $\cM_\Prism(X)$. Thus it follows from the construction in \ref{sub-prismtoshtuka} that the shtuka over $\Spd(\cO_{\breve{E}})$ associated to $T_p(X)$ coincides with the one associated to $X$ by \cite[Ex. 2.3.2]{PR2021}. This proves the claim from the previous paragraph, and completes the proof of the proposition.
\end{proof}

\begin{thm}\label{mainthm}
	Conjecture \ref{conj-PR} holds when $\eG = \eT$ is a torus.
\end{thm}
\begin{proof}
	For the first part of part (a), we can reduce to a single compact open subgroup $\eK = \eK_p\eK^p$, in which case the result follows by the valuative criterion for properness of $\cO_{E_\eK}$ over $\cO_E$. The rest of part (a) follows from the construction of $\sS_\eK$ and Lemma \ref{lem-fmlcomp}. Part (b) follows from Proposition \ref{prop-torusshtuka}, and part (c) follows from Proposition \ref{prop-LT} and the remarks above it.
\end{proof}

%\bibliographystyle{amsalpha}
%\bibliography{Refs}

\providecommand{\bysame}{\leavevmode\hbox to3em{\hrulefill}\thinspace}
\providecommand{\MR}{\relax\ifhmode\unskip\space\fi MR }
\providecommand{\MRhref}[2]{%
	\href{http://www.ams.org/mathscinet-getitem?mr=#1}{#2}
}
\providecommand{\href}[2]{#2}

\end{document}